\documentclass{amsart}[10pt]
\usepackage{etex}
 \usepackage[foot]{amsaddr}
\usepackage{graphicx}
\usepackage{epstopdf}

\usepackage{cite}
\usepackage{tikz, tikz-cd, tkz-graph}

\usepackage{xcolor}

\usepackage{hyperref}
\hypersetup{
  colorlinks,
  linkcolor={red!50!black},
  citecolor={blue!50!black},
  urlcolor={blue!80!black}
}

\usepackage[english]{babel}
\usepackage{amsmath}
\usepackage{amssymb}
\usepackage{amsthm}
\usepackage{xypic}
\usepackage{longtable}
\usepackage{array}

 \usepackage[enableskew]{youngtab}

\usepackage{epsfig}
\usepackage{hyperref}
\usepackage{enumitem}
\usepackage{booktabs}
 \usepackage{longtable} %used by appendix
 \usepackage{pdflscape} % used by appendix
 \usepackage{colortbl} % used by appendix
 \usepackage{arydshln} % used by appendix
 \usepackage{calc} % used by appendix
\usepackage[vcentermath]{genyoungtabtikz}
\Yboxdim{6pt}
\colorlet{lightgray}{gray!25}

\graphicspath{{./images/}}

\newtheorem{thm}{\bf Theorem}[section]
\newtheorem{eg}[thm]{\bf Example}
\newtheorem{prop}[thm]{\bf Proposition}
\newtheorem{cor}[thm]{\bf Corollary}
\newtheorem{rem}[thm]{\bf Remark}
\newtheorem{mydef}[thm]{\bf Definition}
\newtheorem{lem}[thm]{\bf Lemma}

\newcommand{\CC}{\mathbb{C}}
\newcommand{\RR}{\mathbb{R}}
\newcommand{\ZZ}{\mathbb{Z}}

\newcommand{\PP}{\mathbb{P}}

\DeclareMathOperator{\GL}{\mathrm{GL}}
\DeclareMathOperator{\SL}{\mathrm{SL}}

\DeclareMathOperator{\Gr}{\mathrm{Gr}}

\DeclareMathOperator{\Mat}{\mathrm{Mat}}
\newcommand{\Pnr}{\mathcal{P}(n,r)}
\newcommand{\Anr}{\mathcal{A}(n,r)}
\newcommand{\Bnr}{\mathcal{B}(n,r)}
\DeclareMathOperator{\sign}{\mathrm{sign}}
\DeclareMathOperator{\Bl}{\mathrm{Bl}}
\DeclareMathOperator{\LQ}{\mathrm{LQ}}
\DeclareMathOperator{\Cones}{\mathrm{Cones}}
\newcommand{\U}{\"u}

\renewcommand{\emptyset}{\varnothing}

\newtheorem{claim}[thm]{Claim}

\newcolumntype{C}[1]{>{\centering\let\newline\\\arraybackslash\hspace{0pt}}m{#1}}
\title[Unwinding Toric Degenerations]{Unwinding Toric Degenerations and \\ Mirror Symmetry for Grassmannians}
\author{Tom Coates$^{1}$}
\address{$^1$Department of Mathematics\\ Imperial College London}
\email{t.coates@imperial.ac.uk}
\author{Charles Doran$^{2,3}$}
\address{$^{2}$Department of Mathematics\\ University of Alberta}
\address{$^{3}$Center of Mathematical Sciences and Applications\\ Harvard University}

\email{charles.doran@ualberta.ca}
\author{Elana Kalashnikov$^{4}$}
\address{{$^4$Department of Mathematics \\ 
University of Waterloo}}
\email{e2kalash@uwaterloo.ca}

\subjclass[2020]{Primary 14J33; Secondary 14M15, 52B20.}

\begin{document}

\begin{abstract} The most fundamental example of mirror symmetry compares the Fermat hypersurfaces in $\PP^n$ and $\PP^n/G$, where $G$ is a finite group that acts on~$\PP^n$ and preserves the Fermat hypersurface. We generalise this to hypersurfaces in Grassmannians, where the picture is richer and more complex. There is a finite group $G$ that acts on the Grassmannian $\Gr(n,r)$ and preserves an appropriate Calabi--Yau hypersurface. We establish how mirror symmetry, toric degenerations, blow-ups and variation of GIT relate the Calabi--Yau hypersurfaces inside $\Gr(n,r)$ and $\Gr(n,r)/G$. This allows us to describe a compactification of the Eguchi--Hori--Xiong mirror to the Grassmannian, inside a blow-up of the quotient of the Grassmannian by~$G$.
\end{abstract}

\maketitle
\section*{Introduction}
In the most famous and fundamental of examples of mirror symmetry, the Fermat hypersurface in $\PP^{n-1}$ mirrors the Fermat hypersurface in a finite group quotient of $\PP^{n-1}$. The action of the finite group on $\PP^{n-1}$ is obtained by the action of the group
\[\tilde{H}_{n,1}=\langle (\zeta_1,\dots,\zeta_{n}) \mid \zeta_i^n=1, \, \prod \zeta_i=1 \rangle\] 
which is identified with a subgroup of the diagonal matrices of $\SL(\CC^n)$, and hence acts naturally on $\CC^n$. This descends to an action on the GIT quotient $\CC^n/\!/\CC^*=\PP^{n-1}$; we let $H_{n,1}$ denote the quotient group of $\tilde{H}_{n,1}$ that acts effectively on $\PP^{n-1}$. The Fermat hypersurface in $\PP^{n-1}$ is mirror to the Fermat hypersurface in $\PP^{n-1}/H_{n,1}$. This is the first in a series of examples appearing in Greene--Plesser \cite{greene}. 

In this paper, we explore the analogue of Greene--Plesser mirror symmetry for Grassmannians. The Grassmannian of quotients  can be constructed as a GIT quotient $\Mat(r \times n)/\!/\GL(r)$. The group
\begin{equation} \label{eq:fg} \tilde{H}_{n,r}=\langle (\zeta_1,\dots,\zeta_{n}) \mid \zeta_i^n=1, \, (\prod \zeta_i)^r=1 \rangle\end{equation}
is identified with a subgroup of the diagonal matrices in $\GL(\CC^n)$, which acts on $\Mat(r\times n)$ by multiplication on the right. It can also be identified with a subgroup in $\SL(\Mat(r\times n))$ acting  on $\Mat(r \times n)$ as a vector space -- it is in this sense that the $(\prod \zeta_i)^r=1$ is a determinant $1$ condition. The action of $\tilde{H}_{n,r}$ descends to an action on $\Gr(n,r)$; we let $H_{n,r}$ denote the quotient of $\tilde{H}_{n,r}$ that acts effectively on $\Gr(n,r)$. 

We relate $\Gr(n,r)$ and $\Gr(n,r)/{H}_{n,r}$ via mirror symmetry, although -- as suits the Grassmannian -- the story is richer and more complex, involving not just mirror duality but also toric degenerations, a variation of polytope duality, blow-ups, and variation of GIT. In unpacking this picture, we also provide a compactification of the fibers of the Eguchi--Hori--Xiong \cite{eguchi} mirror of the Grassmannian. With a bit of notation, we can state the main result. 

Let $P$ be a lattice polytope and $P^\vee$ its dual. The variety $X_P$ is the toric variety built from the spanning fan of $P$. Mirror pairs of Calabi--Yau hypersurfaces come from pairs of toric varieties $X_P, X_{P^\vee}$. We define the \emph{primitive dual} of a toric variety $X_P$ as  $X_Q$, where $Q$ is the dual polytope of $P$ viewed in the lattice spanned by its vertices (see Definition \ref{def:primitivedual}).  The {primitive dual} of $\PP^{n-1}$ is itself: this explains the re-appearance of $\PP^{n-1}$ (albeit with a group action) under mirror duality.  

In \cite{GL}, Gonciulea and Lakshmibai describe a toric degeneration of the Grassmannian to the toric variety associated to a Gelfand--Cetlin polytope. We call this the Gelfand--Cetlin toric degeneration. 
\begin{thm}\label{thm:intro1}[Theorem \ref{thm:Gsmooth}] Let $X_P$ be the special fiber of the Gelfand--Cetlin toric degeneration, and $X_Q$ its primitive dual. Then there a smoothing of the Batyrev--Borisov dual of $X_Q$ to $\Gr(n,r)/H_{n,r}$.
\end{thm}
The geometry relating $X_P$ and $X_Q$ can be further described.
\begin{thm}\label{thm:intro2}[Theorem \ref{thm:vgit}] There is a toric blow-up of $X_P$ that can be obtained from $X_Q$ by a toric variation of GIT. The reverse holds for their Batyrev--Borisov duals: there is a toric blow-up of $X_{Q^\vee}$ that can be obtained from $X_{P^\vee}$ by a toric variation of GIT.
\end{thm}
The first proposals of mirror symmetry for Grassmannians arise in physics: Hori--Vafa \cite{horivafa} and Eguchi--Hori--Xiong \cite{eguchi}. Hori--Vafa's proposal is rooted in the Abelian/non-Abelian correspondence (later further developed in \cite{CiocanFontanineKimSabbah2008,gusharpe}), while the Eguchi--Hori--Xiong construction can be seen to arise from the Gelfand--Cetlin toric degeneration \cite{conifold, flagdegenerations}. Another more recent proposal is the Marsh--Rietsch Pl\U cker coordinate mirror \cite{MarshRietsch}. All of these proposals focus on a Landau--Ginzburg model, or a superpotential, for the Grassmannian, but do not produce mirror partners for Calabi--Yau hypersurfaces in Grassmannians. However, it is expected that a compactification of the fibers of the superpotential provides a mirror of the Calabi--Yau.  In this paper, we continue this story by exploring mirror symmetry for the Fermat Calabi--Yau hypersurfaces in Grassmannians. 

The two theorems above suggest a possible  mirror to the Fermat Calabi--Yau hypersurface of the Grassmannian, which we describe in \eqref{eq:mirrorproposal} after establishing more notation.  
\begin{thm}\label{thm:intro3}[Theorem \ref{thm:compactification}] The Eguchi--Hori--Xiong Laurent polynomial mirror to the Grassmannian compactifies to an anticanonical hypersurface cut out by \eqref{eq:mirrorproposal} in a blow-up of the Grassmannian  modulo the group $H_{n,r}$. 
\end{thm}
Both the Abelian/non-Abelian correspondence and toric degenerations essentially reduce the question of the mirror symmetry to the toric situation, where older results can be applied. One problem, though, is that it is unclear how to reverse this process on the mirror side - how can we leave the toric context? Our proposal suggests a solution to this problem.     This proposal requires the Gelfand--Cetlin toric degeneration, making extensive use of its combinatorial structure, but the final answer does not depend on the degeneration and it is an intriguing question as to how it compares to the results of Marsh--Rietsch \cite{MarshRietsch} and Rietsch--Williams \cite{RW}.  
\section*{Overview of results}
We now outline the results of the paper.  The first step is a close study of the Batyrev--Borisov mirror of the Gelfand--Cetlin toric degeneration of the Grassmannian. We begin by looking at Batyrev--Borisov mirror symmetry for Fano toric varieties with high Fano index.
\subsection*{High index Fano varieties, lattices, and Batyrev--Borisov mirror symmetry}
Let $P$ be a lattice polytope in a lattice $N$, and let $M$ be the dual lattice. As before, we denote the toric variety arising from the spanning fan $P$  as $X_P$. If $P$ is the polytope of projective space, i.e. if $X_{P}=\PP^n$, then $X_{P^\vee}$ is just $X_P/G$, where $G=(\ZZ/(n+1)\ZZ)^{n-1}$. The group action can be seen to arise from the fact that $\PP^n$ has Fano index $n+1$: the vertices of $P^\vee$ span a sub-lattice $i:\overline{M} \to M$ such that $M /\overline{M} \cong G$.  The fact that $X_{P^\vee}=\PP^{n}/G$ is a consequence of the fact that $i^{-1}(P^\vee) =P$. 

It is a simple corollary of the Batyrev--Borisov construction that if we take intermediate lattices we recover Greene--Plesser mirror symmetry for the Calabi--Yau hypersurface in projective space.  Consider a lattice inclusions $\overline{M} \subset M' \subset M$, and set $(M')^\vee=N'$ and $\overline{M}^\vee=\overline{N}$. Then the anticanonical hypersurface in $\PP^n/(N'/N)$ forms a mirror pair with the anticanonical hypersurface in $\PP^n/(M'/\overline{M}).$ 

We show in \S \ref{sec:lattices} that there is an analogous picture for any Fano toric variety of index $n$. Let $P$ be a reflexive Fano polytope of Fano index $k$. The vertices of $P^\vee$ span a sub-lattice $i:\overline{N} \to N$. Set $G:=N/\overline{N}$. In the projective space example above, and in many other cases, $ G \cong (\ZZ/k\ZZ)^{\dim X_P-1}$. Let $Q$ denote the polytope $i^{-1}(P)$. Then mirror symmetry reverses the roles of $P$ and $Q$:
 \begin{center}
\begin{tikzpicture}\small
\node at (0,0) (A) {$X_P$};
\node at (2,0) (A1) {$X_Q$};
\node at (2,-1) (B) {$X_P/G$};
\node at (0,-1) (B1) {$X_Q/G$};
% \node at (4,-0.5) (C) {\tiny B--B mirror symmetry};
\draw[<->](A)--(B1);
\draw[<->](A1)--(B);
\end{tikzpicture}
\end{center}
where the arrows are Batyrev--Borisov mirror symmetry. We call $Q$ the \emph{primitive dual} of $P$. For polytopes whose vertices generate the ambient lattice,this is a duality.  We also obtain a series of mirror pairs via considering intermediate lattices. 
\begin{prop}
Given lattice inclusions $\overline{N} \subset N' \subset N$,  dualizing to $M \subset M' \subset \overline{M}$, $X_P/(N/N')$ mirrors $X_Q/(M'/\overline{M})$. 
\end{prop}

If $X_P$ is a sufficiently nice toric degeneration of the Grassmannian, then it is in particular a Fano variety of Fano index $n$. If $X_P$ is the Gelfand--Cetlin toric degeneration, then its vertices can be described via the \emph{ladder diagram} (\cite{flagdegenerations}, see \S \ref{sec:GC}). The primitive dual $Q$  can be described via the same diagram. We give a description of $Q$ in this case in \S \ref{sec:GC}.

The series of examples of mirror pairs for projective space includes the quintic twin~\cite{doran}, which intriguingly has the same Picard--Fuchs equations as the quintic. The perspective in this paper, allows us to generalize this `twin phenomenon' to toric degenerations of the Grassmannian. See \S\ref{sec:twin}.

 \subsection*{The Gelfand--Cetlin toric degeneration and VGIT}
 Let $P_{n,r}$ denote the reflexive Fano polytope associated to the Gelfand--Cetlin toric degeneration, and $Q_{n,r}$ its primitive dual. Let $G_{n,r}:=(\ZZ/n\ZZ)^{r(n-r)-1}.$ The polytopes $P_{n,r}$ and $Q_{n,r}$ are closely related via mirror symmetry, but in fact there is another relation. The main result of \S \ref{sec:gt} is Theorem \ref{thm:intro2}: there is a weak Fano toric variety $Y_{n,r}$ that fits into the diagram
 
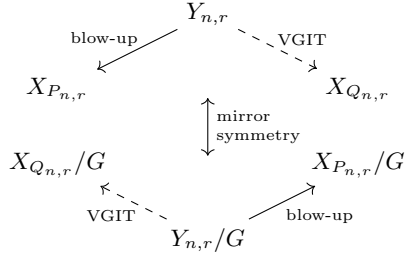
\begin{center}
\begin{tikzpicture}\small
\node at (0,0) (A) {$X_{P_{n,r}}$};
\node at (4,0) (B) {$X_{Q_{n,r}}$};
\node at (2,1) (C) {$Y_{n,r}$};
\draw[dashed, ->](C)--(B)  node[align=right,pos=0.25,right] {\tiny \: VGIT};
\draw[->](C)--(A)  node[pos=0.25,left] {\tiny blow-up \:};
\node at (2,0) (s) {};
\node at (2,-1) (t) {};
\draw[<->](s)--(t) node[align=left,midway,right]{
  \begin{minipage}{1.2cm}
    \tiny mirror symmetry
  \end{minipage}};
\node at (4,-1) (B1) {$X_{P_{n,r}}/G$};
\node at (0,-1) (A1) {$X_{Q_{n,r}}/G$};
\node at (2,-2) (C1) {$Y_{n,r}/G$};
\draw[->](C1)--(B1)  node[pos=0.1,right] {\tiny \: \: blow-up};
\draw[->,dashed](C1)--(A1)  node[align=left,pos=0.1,left] {\tiny VGIT \:\:};
 \end{tikzpicture}
 \end{center}
Note that under mirror symmetry, the roles of the two maps in the diagram are reversed, suggesting the relationship is analogous to that of geometric transitions. The variation of GIT is a toric variation of GIT. While it is not a crepant resolution, the VGIT  is a $K$-equivalence: it preserves the anticanonical class in the sense that the polyhedron associated to the anticanonical divisor is the same for each toric variety. 
\subsection*{Smoothing}
In the first subsection of \S \ref{sec:smoothing}, we prove Theorem \ref{thm:intro1} above, which we can now state more precisely. There is a natural embedding of $X_{P_{n,r}}$ into $\PP^{\binom{n}{r}-1}$ and it is in $\PP^{\binom{n}{r}-1}$ that the Gelfand--Cetlin toric degeneration is seen. The embedding allows $G$ to act not just on $X_{P_{n,r}}$ but on the whole projective space. The heart of the theorem is the remarkable characterization of $H_{n,r}$ as precisely the maximal subgroup of $G$ for which the Grassmannian is an equivariant subvariety (see Theorem \ref{thm:groupiso}). 

It is easy to construct a $G$-equivariant family $Z \subset \PP^{\binom{n}{r}-1} \times \CC^m$, with coefficients $c_i$ on the second factor, such that the special fiber is $X_{P_{n,r}}$ and the general fiber  is $\Gr(n,r)$.  This is just the Gelfand--Cetlin toric degeneration, where we extend the action of $G$ to the second factor in such a way that the varying Pl\U cker relations are $G$-invariant. To obtain a smoothing of $X_{P_{n,r}}/G$, we need to take the quotient of this family by $G$. Fibers of the quotient family are
\[(Z \cap \{c_i^n=b_i\})/G \]
for constants $b_i$. 
\begin{thm} The generic fiber of the quotient family is isomorphic to $\Gr(n,r)/H_{n,r}$.
\end{thm}
We call this a smoothing, as the only singularities arise from the finite group quotient.

We also describe a partial smoothing of $Y_{n,r}$. Using the characterization of $Y_{n,r}$ as the blow-up of the Gelfand--Cetlin degeneration, one can write down binomial equations cutting $Y_{n,r}$ out of an ambient space obtained by blowing-up $\PP^{\binom{n}{r}-1}$ multiple times. This description allows us to reverse the Gelfand--Cetlin degeneration to (partially) smooth to the proper transform of the Grassmannian in the blow-up. 
 
 %To introduce it requires some notation. Let $\Pnr$ denote the set of partitions which fit into an $r\times (n-r)$ box (i.e. those indexing Schubert varieties in $\Gr(n,r)$).  The $\Pnr$ also index the Pl\U cker coordinates on $\Gr(n,r)$. We define a special subset of $\Pnr$ which we call \emph{arrow partitions}. We denote this set $\mathcal{A}(n,r)$ and its complement in $\Pnr$ as $\Bnr$. All arrow partitions consist of a rectangular box of either maximal height or width, together with a single column or row. More precisely, we define $\lambda$ to an arrow partition is $\lambda$ is of one of following two forms:
%\begin{enumerate}
%\item $\lambda=(\underbrace{n-r,\dots,n-r}_{a \text{ times}},b)$ where $0 \leq a <r-1$ and $0 \leq b \leq n-r$.
%\item The transpose of $\lambda$, $\lambda^t$ is $(\underbrace{r,\dots,r}_{a \text{ times}},b)$ where $0< a \leq n-r-1$ and $0 \leq b \leq n-r$.
%\end{enumerate}
%There are $2(r-1)(n-r-1)+n$ arrow partitions. Partitions in $\Bnr$  (i.e. not arrow partitions) are called \emph{excess partitions}. 
%There is a partial order  $\prec$ on $\mathcal{P}(n,r)$. We define two maps $m,M:\mathcal{B}(n,r) \to \Pnr$. Label the excess partitions $\Bnr=\{\beta_1,\dots, \beta_s\}$. Then we define $\Bl \Gr(n,r)$ to be the variety obtained by iteratively blowing-up the loci, for each $\beta_i$,
%\[Z(p_\lambda: m(\beta_i) \prec \lambda \prec M(\beta_i)).\]

\subsection*{Our mirror proposal} 
Consider the equations 
\begin{equation}\label{eq:mirrorproposal} \sum_{\lambda \in \Anr} p_\lambda^n + \psi \prod_{p_\mu \text{ frozen}}p_\mu=0.\end{equation}
The product in the second term is over \emph{frozen} Pl\U cker coordinates which are indexed by partitions $\mu$ which are maximal rectangular partitions. For details and definitions, see \S \ref{sec:mirrorproposal}.  This  equation defines a family of  Calabi--Yau hypersurfaces in $\Bl \Gr(n,r)/{H}$. This is a candidate mirror to a Calabi--Yau hypersurface in $\Gr(n,r)$. In \S \ref{sec:mirrorproposal}, we give evidence towards this conjecture by showing Theorem \ref{thm:intro3} above. 

\section*{Plan of the paper}In \S \ref{sec:lattices}, we explain the special behaviour exhibited by high index Fano varieties under Batyrev--Borisov mirror symmetry and introduce primitive duality. In \S \ref{sec:GC}, we specialize to the Gelfand--Cetlin toric degeneration and describe the combinatorics of $P_{n,r}$ and $Q_{n,r}$. In \S \ref{sec:gt} we describe the geometry relating $X_{P_{n,r}}$ and $X_{Q_{n,r}}$. In \S \ref{sec:smoothing}, we describe  how to `unwind' the toric degenerations after mirror symmetry: we give smoothings for $Y_{n,r}/G$ and $X_{P_{n,r}}/G$. In \S \ref{sec:mirrorproposal}, we introduce a candidate mirror of the anticanonical Fermat Calabi--Yau in the Grassmannian.

\section{Lattices and Batyrev--Borisov mirror symmetry} \label{sec:lattices}
In this first section, we explore Batyrev--Borisov mirror symmetry for high index Fano toric varieties from the perspective of mirror symmetry.

\subsection{Background on toric varieties} 
\subsubsection{Reflexive polytopes} Let $P$ be a full-dimensional lattice polytope in a lattice $M$, containing the origin in its strict interior. Let $N=M^\vee$ be the dual lattice.  The dual of $P$ is
\[P^\vee:=\{v \in N_{\RR} \mid \langle v, w \rangle \geq -1,  \forall w \in P\}.\]
\begin{mydef} A polytope $P$ is reflexive if $P^\vee$ is also a lattice polytope. 
\end{mydef}
Vertices of reflexive polytopes are primitive lattice vectors. 
\begin{mydef} A lattice polytope $P \subset M_\RR$ is \emph{vertex-spanning} if the set of vertices of $P$ span the lattice $M$. 
\end{mydef}
Reflexive polytopes may or may not be vertex-spanning. For example among the 4319~three-dimensional reflexive polytopes, classified in~\cite{KreuzerSkarke98}, 4075 are vertex-spanning and 244 are not.

\subsubsection{Toric varieties from reflexive polytopes} 
Let $P$ be a reflexive polytope, and let $F$ be the spanning fan of $P$. Then a toric variety $X_P$ can be constructed in the usual way from $F$. We say that this is toric variety associated to $P$ (that is, our convention uses spanning rather than dual fans). 
\subsubsection{Toric varieties from weights}
In addition to polytopes and fans, toric varieties can also be constructed using GIT data. The GIT data for a toric variety is
\begin{enumerate}
\item A torus $K \cong (\CC^*)^r$,
\item Weights $D_1,\dots,D_m$ of $K$ defining an action $K$ on $\CC^m$,
\item A stability condition $\omega$, which is a co-character of $K$.
\end{enumerate}
The toric variety $X_\omega$ is then the GIT quotient $\CC^m/\!\!/_\omega K$. For a description of how to pass between the GIT description of a toric variety and the fan construction, see for example \cite[Section 4]{crepantconjecture}.

\subsubsection{Toric divisors} Let $F$ be a fan in a lattice $N$ and let $X$ denote the toric variety defined by $F$. A toric divisor $D$ on $X$ defines a polyhedron $P_D$ in $M_{\RR}$, where $M = N^\vee$; see for example \cite[Chapter 5]{CLS}. This association between toric divisors and polyhedra has the properties that:
\begin{itemize}
  \item the polyhedron for $n D$ is $n P_D$; and 
  \item if two divisors are linearly equivalent then their polyhedra are translates of each other.
\end{itemize}
By definition of duality, if $P$ is a reflexive polytope then $P_{-K_{X_P}}=P^\vee$, where ${-K}_{X_P}$ is the anticanonical divisor given by the sum of the toric divisors on $X_P$.

\subsection{High-index Fano varieties} Throughout this section,  we will fix a $d$-dimensional reflexive Fano polytope $P$ giving rise (via the spanning fan) to a Fano toric variety $X_P$. We assume that $X_P$ has Fano index $n$. Let $N$ denote the ambient lattice of $P$ and $M$ the dual lattice. 

Let $\overline{M}$ be the sublattice of $M$ spanned by the vertices of $P_{-K_{X_P}}$. Let $Q$ denote the lattice polytope in $\overline{M}$ given by the preimage of $P_{-K_{X_P}}$ under the inclusion $\overline{M} \subset M$. 
\begin{mydef}\label{def:primitivedual}The polytope $Q$ is the \emph{primitive dual} of $P$.
\end{mydef}
\begin{rem} The primitive dual of a reflexive polytope is a reflexive polytope. 
\end{rem}
\begin{eg} The polytope of projective space is self-dual under primitive duality. 
\end{eg}
The inclusion of lattices $\overline{M} \subset M$ induces an action of the finite group $M/\overline{M}$ on $X_Q$, and 
\[X_{P^\vee}=X_Q/(M/\overline{M}).\]
We can extend this description of Batyrev--Borisov mirror symmetry to a series of mirror pairs, one for each intermediate lattice.  The inclusion of lattices $\overline{M} \subset M$ defines an inclusion $N \subset \overline{M}^\vee$. We denote $\overline{N}:=\overline{M}^\vee$.    Any intermediate lattice $\overline{M} \subset M_1 \subset M$ defines an intermediate dual lattice 
 \[N \subset N_1:=M_1^\vee \subset \overline{N}.\] 
 \begin{prop} Fix a lattice  $\overline{M} \subset M_1 \subset M$, and let $N_1$ be the lattice dual to $M_1$. Then 
 \[X_P/(N_1/N) \text{ and } X_Q/(M_1/\overline{M})\]
are a Batyrev--Borisov mirror pair. 
 \end{prop}
\begin{proof}
 Let $P_1$ denote the image of $P$ in $N_1$. Then we can describe $X_{P_1}$ as $X_P/(N_1/N)$. The dual $P_1^\vee$ has ambient lattice $M_1$. The vertices of $P_1^\vee$ are identified with that of $Q$ under the inclusion $\overline{M} \subset M_1$, so $X_{P_1^\vee}$ can be constructed as $X_Q/(M_1/\overline{M})$. 
\end{proof}

In the following proposition, we show that primitive duality is a duality, under the assumption that $P$ is vertex-spanning. 
\begin{prop} Let $P$ be a reflexive, vertex-spanning polytope. Let $Q$ be the primitive dual of $P$. Then the primitive dual of $Q$ is $P$. 
\end{prop}
\begin{proof} The dual of the reflexive polytope $Q$ is the image of $P$ under the inclusion of lattice $N \subset \overline{N}$. By assumption, $P$ is vertex-spanning, and so its vertices generate $N$. It follows that $P$ is the primitive dual of $Q$. 
\end{proof}

We can be more explicit given the following assumptions on $P$ and $X_P$:
\begin{enumerate}
\item There is a Cartier toric divisor $D_i$ for some ray $i$ such that $n D_i$ is linearly equivalent to $-K_{X_P}$.
\item Both $P$ and $P_{D_0}$ are vertex-spanning polytopes.
\end{enumerate}
The two polyhedra $P_{n D_i}$  and $P_{-K_{X_P}}$ are translates of each other. Suppose that 
 \[P_{n D_i} +h=P_{-K_{X_P}}\]
 for some lattice element $h \in M$.  
 
 \begin{lem}\label{lem:Gdes} The lattice $\overline{M}$ is generated by $h$ and the generators of $n M$, and
 \[M/\overline{M} \cong (\ZZ/n\ZZ)^{d-1}. \]
 \end{lem}
 \begin{proof}
 As noted, $P^\vee=P_{-K_{X_P}}$. Fix a basis $e_1,\dots,e_d$ of $M$.  The vertices of $P_{n D_0}$ span the lattice given by the basis $n e_j$. It follows the lattice spanned by the vertices of $P^\vee$, that is the lattice $\overline{M}$, is the lattice spanned by $n e_1,\dots,n e_d,h$. Note that as $0$ is a vertex of $P_{D_i}$, $h$ is a vertex of the reflexive polytope $P^\vee$ and hence is a primitive lattice vector. It follows that
 \[M/\overline{M} \cong (\ZZ/n\ZZ)^{d-1}. \]
 \end{proof}
 \section{The Gelfand--Cetlin toric degeneration} \label{sec:GC}
 The goal of this section is recall the construction of the best-known toric degeneration of the Grassmannian, first studied by Gonciulea--Lakshmibai~\cite{GL}. This is a SAGBI basis degeneration of the Grassmannian inside its Pl\U cker embedding (see \cite{combinatorial} for a good exposition). The polytope giving the degenerate toric variety is a Gelfand--Cetlin polytope, so we will call this toric degeneration the Gelfand--Cetlin toric degeneration.  We begin by recalling the details of the construction, including a description of both the polytope and the weight matrix of the toric variety $X_{P_{n,r}}$ that forms the central fiber. We then describe the primitive dual to the polytope $P_{n,r}$, and its weight matrix. 
 
 \subsection{Ladder diagrams}
 The polytope $P_{n,r}$ can be described via a \emph{ladder diagram}. 
 Fix integers $n>r>0$, and set $k:=n-r$. Draw an $r \times k$ grid of boxes. For example, the grid for $n=5$,~$r=2$ is:
 
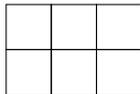
\[\begin{tikzpicture}[scale=0.6]
\draw (0,0) rectangle (1,1);
\draw (0,1) rectangle (1,2);
\draw (1,1) rectangle (2,2);

\draw (1,0) rectangle (2,1);
\draw (2,1) rectangle (3,2);

\draw (2,0) rectangle (3,1);

\end{tikzpicture}\]
Draw vertices on the diagram as follows: at all internal crossing points, and at the bottom left and top right corner. Continuing this example, we obtain:

 
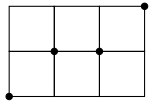
\[\begin{tikzpicture}[scale=0.6]
\draw (0,0) rectangle (1,1);
\draw (0,1) rectangle (1,2);
\draw (1,1) rectangle (2,2);

\draw (1,0) rectangle (2,1);
\draw (2,1) rectangle (3,2);

\draw (2,0) rectangle (3,1);
\draw[fill] (0,0) circle (2pt);
\draw[fill] (1,1) circle (2pt);
\draw[fill] (2,1) circle (2pt);
\draw[fill] (3,2) circle (2pt);

\end{tikzpicture}\]
We consider paths in this diagram between vertices, where travel is always up and to the right. This describes a quiver: vertices are given by the vertices in the ladder diagram, and there is an arrow from vertex $a$ to $b$ for every path in the ladder diagram which does not pass through any other vertex. In the example, the quiver is
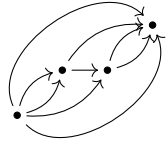
\[
\begin{tikzpicture}[scale=0.6]
\node at (0,0) (A) {};
\node at (1,1) (B) {};
\node at (2,1) (C) {};
\node at (3,2) (D) {};
\draw[fill] (0,0) circle (2pt);
\draw[fill] (1,1) circle (2pt);
\draw[fill] (2,1) circle (2pt);
\draw[fill] (3,2) circle (2pt);
\draw[->] (A) to[bend right] (B);
\draw[->] (A) to[bend left] (B);
\draw[->] (A) to[bend right=80] (D);
\draw[->] (A) to[bend left=80] (D);
\draw[->] (B) to (C);
\draw[->] (A) to[bend right] (C);
\draw[->] (C) to[bend right] (D);
\draw[->] (C) to[bend left] (D);
\draw[->] (B) to[bend left] (D);
\end{tikzpicture}\]

\begin{mydef} We call this quiver the \emph{ladder quiver} $\LQ$, or $\LQ(n,r)$ if $n$ and $r$ are not clear from context. The set of vertices of the ladder quiver is denoted by $\LQ_0$ and the set of arrows by $\LQ_1$. We denote the source and target maps by $s$,~$t:\LQ_1 \to \LQ_0$.  Let $0 \in \LQ_0$ be the source vertex.  Let $\overline{\LQ}_0$ denote the set of vertices with the source vertex omitted. 
\end{mydef}

\begin{rem}
  We henceforth draw the ladder quiver as a ladder diagram, leaving the directions of the arrows implicit. 
\end{rem}

A quiver moduli space is defined by assigning a vector space to each vertex of a quiver, and choosing a stability condition~\cite{King1994}. If we assign a one-dimensional vector space to each vertex of the ladder quiver and choose a Fano stability condition, then we obtain a quiver moduli space that coincides with $X_{P_{n,r}}$: see \cite{kalashnikov2} for details.  

The weight matrix of $X_{P_{n,r}}$ is given by the adjacency matrix of the ladder quiver. More precisely, let $\ZZ^{\overline{\LQ}_0}$ be the lattice with basis $e_v, v \in \overline{\LQ}_0$. For the source vertex $0 \in \LQ_0$, set $e_0=0$. Then the weights are given by $d_a:=-e_{s(a)}+e_{t(a)}$. 

It is often more convenient to write the weights in a different basis. We identify a basis element $f_v$ for each $v \in \overline{LQ}_0$, as follows:
\begin{enumerate}
\item If all arrows with $t(a)=v$ have $s(a)=0$, then we set $f_v:=d_a=e_v$. 
\item If there is a unique arrow such that both $t(a)=v$ and $s(a) \neq 0$, then we set $f_v:=d_a$. 
\item Otherwise, $v$ is the top right corner of a square of vertices. Label the other vertices of the square as shown:
\begin{center}
\begin{tikzpicture}
\node at (0,0) (A) {$v_0$};
\node at (1,0) (B) {$v_1$};
\node at (0,1) (C) {$v_2$};
\node at (1,1) (D) {$v$};
\draw[->](A)--(B);
\draw[->](A)--(C);
\draw[->](B)--(D);
\draw[->](C)--(D);
 \end{tikzpicture}
 \end{center}
then set 
\[f_v:=e_{v}-e_{v_1}-e_{v_2}+e_{v_0}.\]
That is, $f_v$ is a sum $\sum d_a$ where $a$ runs through the arrows drawn in the diagram:
\begin{center}
\begin{tikzpicture}
\node at (0,0) (A) {$v_0$};
\node at (1,0) (B) {$v_1$};
\node at (0,1) (C) {$v_2$};
\node at (1,1) (D) {$v$};
\draw[->](A)--(D);
\draw[<-](A)--(C);
\draw[<-](A)--(B);
 \end{tikzpicture}
 \end{center}
 \end{enumerate}
 In the basis $f_v$,~$v \in \overline{LQ}_0$, the weight matrix 
 \begin{equation}\label{eq:weightmatrix} B:=[b_{v a}]
 \end{equation}
 has rows indexed by $v \in \overline{\LQ}_0$, columns indexed by $a \in \LQ_1$, and all entries are either $1$ or $0$. There are precisely $n$ positive entries in each row. It is easy to identify visually the  set 
 \[\{a \in \LQ_1: b_{v a}=1\}\]
 for each $v$. We describe it by way of example.
 \begin{eg}\label{eg:1set} Let $n=7$ and $r=3$. For each non-source vertex $v$, labeled in red, the arrows satisfying $b_{v a}=1$ are the $n$ arrows fully contained in the red part of the ladder quiver. 
 
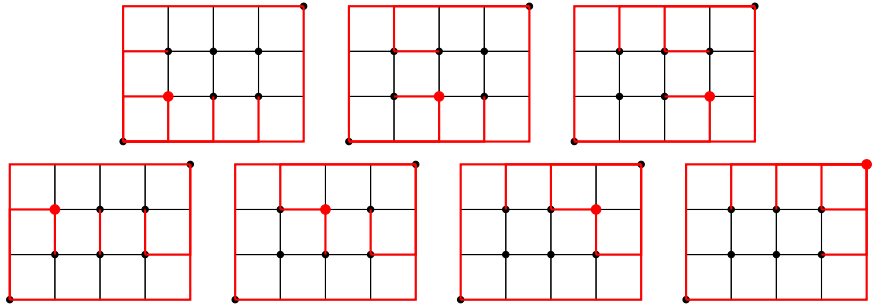
\[\begin{tikzpicture}[scale=0.6]
\draw (0,0) rectangle (1,1);
\draw (0,1) rectangle (1,2);
\draw (0,2) rectangle (1,3);
\draw[] (1,1) rectangle (2,2);
\draw[] (1,0) rectangle (2,1);
\draw[] (1,2) rectangle (2,3);
\draw[] (2,1) rectangle (3,2);
\draw[] (2,2) rectangle (3,3);
\draw[] (2,0) rectangle (3,1);
\draw[] (3,0) rectangle (4,1);
\draw[] (3,1) rectangle (4,2);
\draw[] (3,2) rectangle (4,3);
\draw[fill] (0,0) circle (2pt);
\draw[fill,red] (1,1) circle (3pt);
\draw[fill] (2,1) circle (2pt);
\draw[fill] (3,1) circle (2pt);
\draw[fill] (1,2) circle (2pt);
\draw[fill] (2,2) circle (2pt);
\draw[fill] (3,2) circle (2pt);
\draw[fill] (4,3) circle (2pt);
\draw[red,thick] (0,0)--(0,1)--(1,1);
\draw[red,thick] (0,0)--(1,0)--(1,1);
\draw[red,thick] (0,0)--(0,2)--(1,2);
\draw[red,thick] (0,0)--(0,3)--(4,3);
\draw[red,thick] (0,0)--(2,0)--(2,1);
\draw[red,thick] (0,0)--(3,0)--(3,1);
\draw[red,thick] (0,0)--(4,0)--(4,3);
\end{tikzpicture}
\hspace{5mm}
\begin{tikzpicture}[scale=0.6]
\draw (0,0) rectangle (1,1);
\draw (0,1) rectangle (1,2);
\draw (0,2) rectangle (1,3);
\draw[] (1,1) rectangle (2,2);
\draw[] (1,0) rectangle (2,1);
\draw[] (1,2) rectangle (2,3);
\draw[] (2,1) rectangle (3,2);
\draw[] (2,2) rectangle (3,3);
\draw[] (2,0) rectangle (3,1);
\draw[] (3,0) rectangle (4,1);
\draw[] (3,1) rectangle (4,2);
\draw[] (3,2) rectangle (4,3);
\draw[fill] (0,0) circle (2pt);
\draw[fill] (1,1) circle (2pt);
\draw[fill,red] (2,1) circle (3pt);
\draw[fill] (3,1) circle (2pt);
\draw[fill] (1,2) circle (2pt);
\draw[fill] (2,2) circle (2pt);
\draw[fill] (3,2) circle (2pt);
\draw[fill] (4,3) circle (2pt);
\draw[red,thick] (1,1)--(2,1);
\draw[red,thick] (1,2)--(2,2);
\draw[red,thick] (1,2)--(1,3)--(4,3);
\draw[red,thick] (0,0)--(2,0)--(2,1);
\draw[red,thick] (0,0)--(3,0)--(3,1);
\draw[red,thick] (0,0)--(4,0)--(4,3);
\draw[red,thick] (0,0)--(0,3)--(4,3);
\end{tikzpicture}
\hspace{5mm}
\begin{tikzpicture}[scale=0.6]
\draw (0,0) rectangle (1,1);
\draw (0,1) rectangle (1,2);
\draw (0,2) rectangle (1,3);
\draw[] (1,1) rectangle (2,2);
\draw[] (1,0) rectangle (2,1);
\draw[] (1,2) rectangle (2,3);
\draw[] (2,1) rectangle (3,2);
\draw[] (2,2) rectangle (3,3);
\draw[] (2,0) rectangle (3,1);
\draw[] (3,0) rectangle (4,1);
\draw[] (3,1) rectangle (4,2);
\draw[] (3,2) rectangle (4,3);
\draw[fill] (0,0) circle (2pt);
\draw[fill] (1,1) circle (2pt);
\draw[fill] (2,1) circle (2pt);
\draw[fill,red] (3,1) circle (3pt);
\draw[fill] (1,2) circle (2pt);
\draw[fill] (2,2) circle (2pt);
\draw[fill] (3,2) circle (2pt);
\draw[fill] (4,3) circle (2pt);
\draw[red,thick] (0,0)--(4,0)--(4,3);
\draw[red,thick] (0,0)--(0,3)--(4,3);
\draw[red,thick] (2,1)--(3,1);
\draw[red,thick] (2,2)--(3,2);
\draw[red,thick] (2,2)--(2,3)--(4,3);
\draw[red,thick] (0,0)--(3,0)--(3,1);
\draw[red,thick] (1,2)--(1,3)--(4,3);
\end{tikzpicture}
\]
 \[\begin{tikzpicture}[scale=0.6]
\draw (0,0) rectangle (1,1);
\draw (0,1) rectangle (1,2);
\draw (0,2) rectangle (1,3);
\draw[] (1,1) rectangle (2,2);
\draw[] (1,0) rectangle (2,1);
\draw[] (1,2) rectangle (2,3);
\draw[] (2,1) rectangle (3,2);
\draw[] (2,2) rectangle (3,3);
\draw[] (2,0) rectangle (3,1);
\draw[] (3,0) rectangle (4,1);
\draw[] (3,1) rectangle (4,2);
\draw[] (3,2) rectangle (4,3);
\draw[fill] (0,0) circle (2pt);
\draw[fill] (1,1) circle (2pt);
\draw[fill] (2,1) circle (2pt);
\draw[fill] (3,1) circle (2pt);
\draw[fill,red] (1,2) circle (3pt);
\draw[fill] (2,2) circle (2pt);
\draw[fill] (3,2) circle (2pt);
\draw[fill] (4,3) circle (2pt);
\draw[red,thick] (1,2)--(1,1);
\draw[red,thick] (2,2)--(2,1);
\draw[red,thick] (3,2)--(3,1);
\draw[red,thick] (3,1)--(4,1)--(4,3);
\draw[red,thick] (0,0)--(0,2)--(1,2);
\draw[red,thick] (0,0)--(4,0)--(4,3);
\draw[red,thick] (0,0)--(0,3)--(4,3);
\end{tikzpicture}
\hspace{5mm}
\begin{tikzpicture}[scale=0.6]
\draw (0,0) rectangle (1,1);
\draw (0,1) rectangle (1,2);
\draw (0,2) rectangle (1,3);
\draw[] (1,1) rectangle (2,2);
\draw[] (1,0) rectangle (2,1);
\draw[] (1,2) rectangle (2,3);
\draw[] (2,1) rectangle (3,2);
\draw[] (2,2) rectangle (3,3);
\draw[] (2,0) rectangle (3,1);
\draw[] (3,0) rectangle (4,1);
\draw[] (3,1) rectangle (4,2);
\draw[] (3,2) rectangle (4,3);
\draw[fill] (0,0) circle (2pt);
\draw[fill] (1,1) circle (2pt);
\draw[fill] (2,1) circle (2pt);
\draw[fill] (3,1) circle (2pt);
\draw[fill] (1,2) circle (2pt);
\draw[fill,red] (2,2) circle (3pt);
\draw[fill] (3,2) circle (2pt);
\draw[fill] (4,3) circle (2pt);
\draw[red,thick] (0,0)--(4,0)--(4,3);
\draw[red,thick] (0,0)--(0,3)--(4,3);
\draw[red,thick] (1,2)--(2,2);
\draw[red,thick] (2,2)--(2,1);
\draw[red,thick] (3,2)--(3,1);
\draw[red,thick] (3,1)--(4,1)--(4,3);
\draw[red,thick] (1,2)--(1,3)--(4,3);
\end{tikzpicture}
\hspace{5mm}
\begin{tikzpicture}[scale=0.6]
\draw (0,0) rectangle (1,1);
\draw (0,1) rectangle (1,2);
\draw (0,2) rectangle (1,3);
\draw[] (1,1) rectangle (2,2);
\draw[] (1,0) rectangle (2,1);
\draw[] (1,2) rectangle (2,3);
\draw[] (2,1) rectangle (3,2);
\draw[] (2,2) rectangle (3,3);
\draw[] (2,0) rectangle (3,1);
\draw[] (3,0) rectangle (4,1);
\draw[] (3,1) rectangle (4,2);
\draw[] (3,2) rectangle (4,3);
\draw[fill] (0,0) circle (2pt);
\draw[fill] (1,1) circle (2pt);
\draw[fill] (2,1) circle (2pt);
\draw[fill] (3,1) circle (2pt);
\draw[fill] (1,2) circle (2pt);
\draw[fill] (2,2) circle (2pt);
\draw[fill,red] (3,2) circle (3pt);
\draw[fill] (4,3) circle (2pt);
\draw[red,thick] (0,0)--(4,0)--(4,3);
\draw[red,thick] (0,0)--(0,3)--(4,3);
\draw[red,thick] (2,2)--(3,2);
\draw[red,thick] (2,2)--(2,3)--(4,3);
\draw[red,thick] (3,2)--(3,1);
\draw[red,thick] (3,1)--(4,1)--(4,3);
\draw[red,thick] (1,2)--(1,3)--(4,3);
\end{tikzpicture}
\hspace{5mm}
\begin{tikzpicture}[scale=0.6]
\draw (0,0) rectangle (1,1);
\draw (0,1) rectangle (1,2);
\draw (0,2) rectangle (1,3);
\draw[] (1,1) rectangle (2,2);
\draw[] (1,0) rectangle (2,1);
\draw[] (1,2) rectangle (2,3);
\draw[] (2,1) rectangle (3,2);
\draw[] (2,2) rectangle (3,3);
\draw[] (2,0) rectangle (3,1);
\draw[] (3,0) rectangle (4,1);
\draw[] (3,1) rectangle (4,2);
\draw[] (3,2) rectangle (4,3);
\draw[fill] (0,0) circle (2pt);
\draw[fill] (1,1) circle (2pt);
\draw[fill] (2,1) circle (2pt);
\draw[fill] (3,1) circle (2pt);
\draw[fill] (1,2) circle (2pt);
\draw[fill] (2,2) circle (2pt);
\draw[fill] (3,2) circle (2pt);
\draw[fill,red] (4,3) circle (3pt);
\draw[red,thick] (0,0)--(4,0)--(4,3);
\draw[red,thick] (0,0)--(0,3)--(4,3);
\draw[red,thick] (3,2)--(4,2)--(4,3);
\draw[red,thick] (3,2)--(3,3)--(4,3);
\draw[red,thick] (2,2)--(2,3)--(4,3);
\draw[red,thick] (1,2)--(1,3)--(4,3);
\draw[red,thick] (3,1)--(4,1)--(4,3);
\end{tikzpicture}
\]

 \end{eg}
As one can see from the example, for a non-source vertex $v$, $b_{v a}=1$ if $a$ is an arrow that either
\begin{itemize}
\item has a horizontal step above and in the same column as the horizontal arrow with target $v$; or
\item has a vertical step to the right and in the same row as the vertical arrow with target $v$. 
\end{itemize}
 In particular, every path from the bottom left to top right corner contains exactly one of the arrows in
  \[\{a \in \LQ_1: b_{v a}=1\}\]
  for each $v$.

A great deal of the geometry of the special fiber $X_{P_{n,r}}$ can be read off the ladder quiver. For example, collections of arrows give divisors.  The toric variety $X_{P_{n,r}}$ has Picard rank~1, and the Picard lattice is generated by a divisor corresponding to a path from the bottom left vertex to the top right; each such path is linearly equivalent to the ample generator for the Picard lattice. This generator is the ample line bundle that gives the embedding of $X_{P_{n,r}}$ into $\PP^{\binom{n}{r}-1}$, where it is the special fiber of the Gelfand--Cetlin degeneration. As such, we label it $\mathcal{O}(1)$. The space of sections of $\mathcal{O}(1)$ has dimension $\binom{n}{r}$, and basis indexed by paths from the bottom left vertex to the top right. Fixing such a path~$p$, we can consider the area above~$p$: this gives a partition $\lambda$ which fits into an $r \times k$ box. We label the corresponding section of $\mathcal{O}(1)$ by $s_\lambda$. The Cox ring is a polynomial ring $\CC[y_a: a \in \LQ_1]$, and  $s_\lambda = \prod_{a \in p} y_a$.   

Partitions fitting into an $r \times k$ box also index Pl\U cker coordinates. Given a partition $\lambda$, there are $n$ horizontal and vertical steps in the corresponding path~$p$, and the location of the vertical steps gives a size-$r$ subset of $\{1,\dots,n\}$.

\begin{mydef}\label{def:indexing}  We denote by $\Pnr$ the set of partitions $\lambda$ that fit into an $r \times k$ box. If $\lambda \in \Pnr$, we define $I_\lambda$ to be the corresponding size-$r$ subset of $\{1,\dots,n\}$.
\end{mydef}

\begin{eg} The ladder diagram for $Gr(7,3)$ is below. The path in blue describes the partition $(2,1)$. This indexes the Pl\U cker coordinate $p_{\{1,3,5\}}$. 
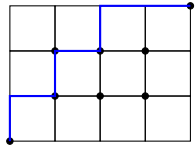
\[\begin{tikzpicture}[scale=0.6]
\draw (0,0) rectangle (1,1);
\draw (0,1) rectangle (1,2);
\draw (0,2) rectangle (1,3);
\draw[] (1,1) rectangle (2,2);
\draw[] (1,0) rectangle (2,1);
\draw[] (1,2) rectangle (2,3);
\draw[] (2,1) rectangle (3,2);
\draw[] (2,2) rectangle (3,3);
\draw[] (2,0) rectangle (3,1);
\draw[] (3,0) rectangle (4,1);
\draw[] (3,1) rectangle (4,2);
\draw[] (3,2) rectangle (4,3);
\draw[fill] (0,0) circle (2pt);
\draw[fill] (1,1) circle (2pt);
\draw[fill] (2,1) circle (2pt);
\draw[fill] (3,1) circle (2pt);
\draw[fill] (1,2) circle (2pt);
\draw[fill] (2,2) circle (2pt);
\draw[fill] (3,2) circle (2pt);
\draw[fill] (4,3) circle (2pt);

\draw[blue,thick] (0,0)--(0,1)--(1,1)--(1,2)--(2,2)--(2,3)--(4,3);
\end{tikzpicture}\]

\end{eg}
\subsection{The polytope $P_{n,r}$ and its primitive dual}

To describe the Gelfand--Cetlin toric degeneration via a polytope, we draw another quiver dual to the ladder quiver (this is the quiver that gives the head-over-tails mirror \cite{eguchi}). There is a vertex for the dual quiver in each square of the ladder diagram, as well as two external vertices above and to the right of the ladder diagram. Arrows move down and to the left, so we obtain:
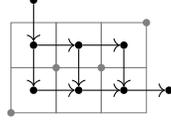
\[\begin{tikzpicture}[scale=0.6]
\draw[gray] (0,0) rectangle (1,1);
\draw[gray] (0,1) rectangle (1,2);
\draw[gray] (1,1) rectangle (2,2);

\draw[gray] (1,0) rectangle (2,1);
\draw[gray] (2,1) rectangle (3,2);

\draw[gray] (2,0) rectangle (3,1);
\draw[fill, gray] (0,0) circle (2pt);
\draw[fill, gray] (1,1) circle (2pt);
\draw[fill, gray] (2,1) circle (2pt);
\draw[fill, gray] (3,2) circle (2pt);

\node[circle,fill,scale=0.3] at (0.5,0.5) (a) {};
\node[circle,fill,scale=0.3] at (1.5,0.5) (b) {};
\node[circle,fill,scale=0.3] at (2.5,0.5) (c) {};
\node[circle,fill,scale=0.3] at (3.5,0.5) (d) {};
\node[circle,fill,scale=0.3] at (0.5,1.5) (e) {};
\node[circle,fill,scale=0.3] at (0.5,2.5) (f) {};

\node[circle,fill,scale=0.3] at (1.5,1.5) (i) {};

\node[circle,fill,scale=0.3] at (2.5,1.5) (k) {};

\draw[<-] (d)--(c);
\draw[<-] (c)--(b);
\draw[<-] (b)--(a);

\draw[<-] (k)--(i);
\draw[<-] (i)--(e);

\draw[->] (e)--(a);
\draw[->] (f)--(e);

\draw[->] (i)--(b);

\draw[->] (k)--(c);

\end{tikzpicture}\]
Each arrow in the ladder quiver is crossed by precisely one arrow in the dual quiver.

Notice that there are $k r$ interior vertices. Let $M$ be an $k r$-dimensional lattice, where we index the basis $e'_v$ by interior vertices $v$ in the dual quiver. Set $e'_v=0$ for external vertices $v$.
\begin{mydef} The polytope $P_{n,r}$ has a vertex for each arrow in the dual quiver. The vertex for arrow $a$ is $w_{a}:=-e'_{s(a)}+e'_{t(a)} \in M$.
\end{mydef}
Sometimes instead of the basis $e'_v$ it is convenient to use another basis, given by arrows. For each internal vertex $v$ in the dual quiver, we assign an arrow $a_v$. If there is a horizontal arrow out of vertex $v$, then this is $a_v$; otherwise $a_v$ is the vertical arrow out of $v$.  Then we set $b_v:=w_{a_v}$. In the diagram below, we highlight the arrows $a_v$ in red for $n=7, r=3$. 
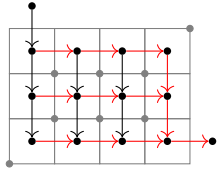
\[\begin{tikzpicture}[scale=0.6]
\draw[gray] (0,0) rectangle (1,1);
\draw[gray] (0,1) rectangle (1,2);
\draw[gray] (0,2) rectangle (1,3);
\draw[gray] (1,1) rectangle (2,2);
\draw[gray] (1,0) rectangle (2,1);
\draw[gray] (1,2) rectangle (2,3);
\draw[gray] (2,1) rectangle (3,2);
\draw[gray] (2,2) rectangle (3,3);
\draw[gray] (2,0) rectangle (3,1);
\draw[gray] (3,0) rectangle (4,1);
\draw[gray] (3,1) rectangle (4,2);
\draw[gray] (3,2) rectangle (4,3);
\draw[fill, gray] (0,0) circle (2pt);
\draw[fill, gray] (1,1) circle (2pt);
\draw[fill, gray] (2,1) circle (2pt);
\draw[fill, gray] (3,1) circle (2pt);
\draw[fill, gray] (1,2) circle (2pt);
\draw[fill, gray] (2,2) circle (2pt);
\draw[fill, gray] (3,2) circle (2pt);
\draw[fill, gray] (4,3) circle (2pt);

\node[circle,fill,scale=0.3] at (0.5,0.5) (a) {};
\node[circle,fill,scale=0.3] at (1.5,0.5) (b) {};
\node[circle,fill,scale=0.3] at (2.5,0.5) (c) {};
\node[circle,fill,scale=0.3] at (3.5,0.5) (d) {};
\node[circle,fill,scale=0.3] at (4.5,0.5) (l) {};

\node[circle,fill,scale=0.3] at (0.5,1.5) (e) {};
\node[circle,fill,scale=0.3] at (1.5,1.5) (i) {};
\node[circle,fill,scale=0.3] at (2.5,1.5) (k) {};
\node[circle,fill,scale=0.3] at (3.5,1.5) (m) {};

\node[circle,fill,scale=0.3] at (0.5,2.5) (n) {};
\node[circle,fill,scale=0.3] at (1.5,2.5) (o) {};
\node[circle,fill,scale=0.3] at (2.5,2.5) (p) {};
\node[circle,fill,scale=0.3] at (3.5,2.5) (q) {};

\node[circle,fill,scale=0.3] at (0.5,3.5) (f) {};

\draw[red,<-] (l) --(d);
\draw[red,<-] (d)--(c);
\draw[red,<-] (c)--(b);
\draw[red,<-] (b)--(a);

\draw[red,<-] (k)--(i);
\draw[red,<-] (i)--(e);
\draw[red,<-] (m)--(k);

\draw[red,<-] (o)--(n);
\draw[red,<-] (p)--(o);
\draw[red,<-] (q)--(p);

\draw[->] (e)--(a);
\draw[->] (n)--(e);
\draw[->] (f)--(n);

\draw[->] (i)--(b);
\draw[->] (o)--(i);

\draw[->] (k)--(c);
\draw[->] (p)--(k);

\draw[red,->] (m)--(d);
\draw[red,->] (q)--(m);
\end{tikzpicture}
\]

Let $N$ be the lattice dual to $M$, with dual bases ${e'}_v^\vee$ and $b_v^\vee$.   We now construct the primitive dual of $P_{n,r}$, which we call $Q_{n,r}$. By definition, this is the pullback of the anticanonical polytope of $X_{P_{n,r}}$ to the lattice spanned by its vertices. We construct it in three stages:
\begin{enumerate}
\item Construct the polytope for the divisor given by the path corresponding to the empty partition. 
\item Compare the $n$-th dilate of this polytope with anticanonical polytope of $X_{P_{n,r}}$. 
\item Compute the sub-lattice generated by the vertices of the polytope. 
\end{enumerate}
Let $D_{\emptyset}$ denote the divisor corresponding to the empty partition. The path corresponding to the empty set consists of exactly one arrow, which we denote $a_\emptyset$. Set 
\begin{equation} \label{eqn:mlambda} m_\lambda:=\sum_{\lambda \text{ crosses } a_v} b^\vee_v.\end{equation}

\begin{lem}\label{lem:dualpairing} Let $a$ be an arrow in the dual quiver that is not $a_\emptyset$, and $\lambda$ a path that does not correspond to $\emptyset$. Then 
\[\langle m_\lambda,w_a\rangle=1\] 
if $\lambda$ crosses $a$ and is $0$ otherwise.  Then 
\[\langle m_\lambda,w_a\rangle=1\] 
if $\lambda$ crosses $a$ and is $0$ otherwise.  
\end{lem}
\begin{proof}
To compute, we consider both vectors written in the $b_{v}$ or $b_v^\vee$ basis. The statement is obvious if $a=a_v$ for some $v$. Otherwise, we can assume that $a$ is a vertical arrow with both source and target internal vertices. The vector $w_a$ is expanded in the $b_v$  by following the path starting at $s(a)$, going all the way to the right, then down, and then going in reverse back to $t(a)$.  Then $w_a$ is the sum of these $b_v$, where $b_v$ has coefficient $\pm 1$ depending on whether it followed in the positive direction or in reverse.  Note that the number of arrows followed in the positive direction is one more than the number of arrows in the reverse direction. 

Suppose $\lambda$ crosses $a$. After crossing $a$, it must cross exactly one of the arrows going in the positive direction, so $\langle m_\lambda,w_a\rangle=1$. If $\lambda$ does not cross $a$, it either avoids the arrows used in the expansion of $w_a$ entirely, or it crosses one arrow in the reverse direction and one arrow in the positive, so 
\[\langle m_\lambda,w_a\rangle=1-1=0.\]

For a illustration, consider the following example:
\[\begin{tikzpicture}[scale=0.6]
\draw[gray] (0,0) rectangle (1,1);
\draw[gray] (0,1) rectangle (1,2);
\draw[gray] (0,2) rectangle (1,3);
\draw[gray] (1,1) rectangle (2,2);
\draw[gray] (1,0) rectangle (2,1);
\draw[gray] (1,2) rectangle (2,3);
\draw[gray] (2,1) rectangle (3,2);
\draw[gray] (2,2) rectangle (3,3);
\draw[gray] (2,0) rectangle (3,1);
\draw[gray] (3,0) rectangle (4,1);
\draw[gray] (3,1) rectangle (4,2);
\draw[gray] (3,2) rectangle (4,3);
\draw[fill, gray] (0,0) circle (2pt);
\draw[fill, gray] (1,1) circle (2pt);
\draw[fill, gray] (2,1) circle (2pt);
\draw[fill, gray] (3,1) circle (2pt);
\draw[fill, gray] (1,2) circle (2pt);
\draw[fill, gray] (2,2) circle (2pt);
\draw[fill, gray] (3,2) circle (2pt);
\draw[fill, gray] (4,3) circle (2pt);

\draw[blue,thick] (0,0)--(0,1)--(1,1)--(1,2)--(2,2)--(2,3)--(4,3);

\node[circle,fill,scale=0.3] at (0.5,0.5) (a) {};
\node[circle,fill,scale=0.3] at (1.5,0.5) (b) {};
\node[circle,fill,scale=0.3] at (2.5,0.5) (c) {};
\node[circle,fill,scale=0.3] at (3.5,0.5) (d) {};
\node[circle,fill,scale=0.3] at (4.5,0.5) (l) {};

\node[circle,fill,scale=0.3] at (0.5,1.5) (e) {};
\node[circle,fill,scale=0.3] at (1.5,1.5) (i) {};
\node[circle,fill,scale=0.3] at (2.5,1.5) (k) {};
\node[circle,fill,scale=0.3] at (3.5,1.5) (m) {};

\node[circle,fill,scale=0.3] at (0.5,2.5) (n) {};
\node[circle,fill,scale=0.3] at (1.5,2.5) (o) {};
\node[circle,fill,scale=0.3] at (2.5,2.5) (p) {};
\node[circle,fill,scale=0.3] at (3.5,2.5) (q) {};

\node[circle,fill,scale=0.3] at (0.5,3.5) (f) {};

\draw[red,<-] (l) --(d);
\draw[red,<-] (d)--(c);
\draw[red,<-] (c)--(b);
\draw[red,<-] (b)--(a);

\draw[red,<-] (k)--(i);
\draw[red,<-] (i)--(e);
\draw[red,<-] (m)--(k);

\draw[red,<-] (o)--(n);
\draw[red,<-] (p)--(o);
\draw[red,<-] (q)--(p);

\draw[->] (e)--(a);
\draw[->] (n)--(e);
\draw[->] (f)--(n);

\draw[->] (i)--(b);
\draw[->] (o)--(i);

\draw[->] (k)--(c);
\draw[->] (p)--(k);

\draw[red,->] (m)--(d);
\draw[red,->] (q)--(m);
\end{tikzpicture}\]
\end{proof}

\begin{cor}
The polytope $P_{D_{\emptyset}}$ corresponding to the divisor $D_\emptyset$ is the convex hull of the points
\[ \left\{m_\lambda:=\sum_{\lambda \text{ crosses } a_v} b^\vee_v \mid \lambda \in \Pnr \right\}.\]
\end{cor}
\begin{proof}
One easy way to see the claim is to check that points in this polytope homogenize to sections of $D_{\emptyset}$. This is the case if $m_{\emptyset}$ is the origin and for all other partitions $\lambda$:
\begin{itemize}
\item $\langle m_{\lambda}, w_{a_\emptyset} \rangle = -1$,
\item $\langle m_{\lambda}, w_a \rangle =1$ if $a$ is in the path $\lambda$ and $\langle m_{\lambda}, w_a \rangle =0$ otherwise. 
 \end{itemize}
The first condition is easy to check visually, and the second is precisely  Lemma \ref{lem:dualpairing}.

\end{proof}
To compute $P_{n,r}^\vee$, one must dilate the polyhedron $P_{D_\emptyset}$ by a factor of $n$, and then perform a linear shift. That is,
\[P_{n,r}^\vee=n P_{D_\emptyset}-h,\]
where $h=\sum_{v} b_v$. This can be seen by comparing the homogenization procedure of $P_{n D_\emptyset}$ and $P_{-\sum D_a}=P_{n,r}^\vee$. It follows that the vertices of $P_{n,r}^\vee$ are
\[n \cdot m_\lambda-h.\]
Finally, we note that the lattice $\overline{M}$ generated by $h, \{n b_v: \text{$v$ an interior vertex}\}$ is the lattice generated by the vertices of $P_{n,r}^\vee$. The pre-image of $P_{n,r}^\vee$ under the inclusion of lattice $\overline{M} \subset M$  is by definition the primitive dual $Q_{n,r}$. This is a Fano reflexive polytope with $\binom{n}{r}$ vertices, labeled by partitions $\lambda \in \Pnr.$ We denote the vertices of $Q_{n,r}$ as $v_\lambda$, $\lambda \in \Pnr$. 
\subsection{Weights of $X_{Q_{n,r}}$}
As $Q_{n,r}$ is a Fano reflexive polytope, we can consider the Fano toric variety $X_{Q_n,r}$. We will describe a spanning set of the kernel of the ray sequence (a basis for the kernel would give the rows of a weight matrix). There is a relation for each \emph{$m$-covering} of the ladder quiver $\LQ$. 
\begin{mydef} An \emph{$m$-covering} of the ladder quiver $\LQ$ is a collection of partitions such that the collection of associated paths in the ladder quiver contains each arrow precisely $m$ times.  
\end{mydef}
The $1$-coverings are given by \emph{crossing diagrams}. 
\begin{mydef} A \emph{crossing diagram} of $\LQ_{n,r}$ is an assignment of $X$ or $O$ to each internal vertex of $\LQ_{n,r}$. 
\end{mydef}
For example, the following is a crossing diagram for $\LQ_{5,2}$. 
 
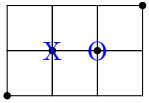
\[\begin{tikzpicture}[scale=0.6]
\draw (0,0) rectangle (1,1);
\draw (0,1) rectangle (1,2);
\draw (1,1) rectangle (2,2);

\draw (1,0) rectangle (2,1);
\draw (2,1) rectangle (3,2);

\draw (2,0) rectangle (3,1);
\draw[fill] (0,0) circle (2pt);
\draw[fill] (1,1) circle (2pt);
\draw[fill] (2,1) circle (2pt);
\draw[fill] (3,2) circle (2pt);
\node[blue] at (1,1) (A) {X};
\node[blue] at (2,1) (B) {O};
\end{tikzpicture}\]
A crossing diagram defines a set of $n$ partitions in $\Pnr$. Each partition corresponds to a path in the ladder quiver. A path $\lambda$ is a path for a given crossing diagram if
\begin{itemize}
\item For every vertex labeled $X$, if the path $\lambda$ contains this vertex, it goes straight through this vertex. 
\item For every vertex labeled $O$, if the path $\lambda$ contains this vertex, it makes a $90^\circ$ turn at this vertex.
\end{itemize}
If we continue the same example as above, the paths are $\emptyset, \yng(2), \yng(1,1),\yng(3,2),\yng(3,3)$. Notice that the empty partition and the maximal partition $(k^r)$ are paths for every crossing diagram. Together, the paths of a fixed crossing diagram cover every arrow in the ladder quiver precisely once, so they are a $1$-covering. 
\begin{prop} Let $\{\lambda_i\}$ be the $m n$ paths of an $m$-covering of a ladder quiver. Then $\sum_{i=1}^{mn} v_{\lambda_i}=0.$ These relations span all relations amongst the $v_\lambda$. 
\end{prop}
\begin{proof}
To show the first part of the proposition, note that though the $v_{\lambda}$ are vectors in the sublattice $\overline{M}$ spanned by the vertices of $Q_{n,r}$,  it is equivalent to show that  
\[
  \sum_{i=1}^{mn} i(v_{\lambda_i})=0, \qquad \text{that is,} \qquad \sum_{i=1}^{mn} (n m_{\lambda_i}-h) = 0.
\]
Here $i$ is the inclusion $i:\overline{M} \to M$. 

Recall that $m_\lambda:=\sum_{\lambda \text{ crosses } a_v} b^\vee_v$. Since the collection of paths $\{\lambda_i\}$ together cover each arrow in $\LQ$ exactly $m$ times, it follows that 
\[\sum_{i=1}^{mn} m_{\lambda_i}=\sum_{v} b_v^\vee= m h,\]
which is equivalent to the desired equality. 

For the second part of the proposition, suppose there is some collection $\{\mu_1,\dots,\mu_k\}$ and coefficients $a_i$ satisfying
\[\sum_{i=1}^k a_i v_{\mu_i}=0.\]
Note that we can assume, by adding multiples of the relations obtained from the crossing diagrams if necessary, that all the $a_i$ are non-negative. In fact, if we allow partitions to appear multiple times, we can assume that the $a_i=1$.

As above, we convert this to an equality in $N$, where we see that
\[\sum_{i=1}^k (n m_{\mu_i}-h)=0.\]
It follows that $k$ is a multiple of $n$ and
\[\sum_{i=1}^k m_{\mu_i}=\frac{k}{n} h.\]
If we re-interpret this as a statement about ladder diagrams, it says that for each basis arrow $a_v$ in the dual quiver, precisely $k/n$ of the paths $\mu_i$ cross $a_v$. 

In fact, this is enough to show that precisely $k/n$ of the paths $\mu_i$ cross \emph{any} arrow in the dual quiver. To see this, first consider a non-basis arrow $a$, $a \neq a_\emptyset$. It is a vertical arrow with internal source and target (in the dual quiver). There is a rectangle of arrows in the dual diagram, with vertical arrow $a$ on the left side, and all other arrows (on the boundary) basis arrows. For example, this is such a rectangle (where $a$ is the green arrow and the red arrows are -- as before -- the $a_v$):  
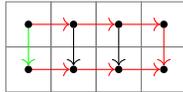
\[\begin{tikzpicture}[scale=0.6]
\draw[gray] (0,0) rectangle (1,1);
\draw[gray] (0,1) rectangle (1,2);
\draw[gray] (1,1) rectangle (2,2);
\draw[gray] (1,0) rectangle (2,1);
\draw[gray] (2,1) rectangle (3,2);
\draw[gray] (2,0) rectangle (3,1);
\draw[gray] (3,1) rectangle (4,2);
\draw[gray] (3,0) rectangle (4,1);

\node[circle,fill,scale=0.3] at (0.5,0.5) (a) {};
\node[circle,fill,scale=0.3] at (1.5,0.5) (b) {};
\node[circle,fill,scale=0.3] at (2.5,0.5) (c) {};
\node[circle,fill,scale=0.3] at (3.5,0.5) (d) {};

\node[circle,fill,scale=0.3] at (0.5,1.5) (e) {};
\node[circle,fill,scale=0.3] at (1.5,1.5) (i) {};
\node[circle,fill,scale=0.3] at (2.5,1.5) (k) {};
\node[circle,fill,scale=0.3] at (3.5,1.5) (l) {};

\draw[red,<-] (d)--(c);
\draw[red,<-] (c)--(b);
\draw[red,<-] (b)--(a);

\draw[red,<-] (l)--(k);
\draw[red,<-] (k)--(i);
\draw[red,<-] (i)--(e);

\draw[green,->] (e)--(a);

\draw[->] (i)--(b);
\draw[->] (k)--(c);

\draw[red,->] (l)--(d);

\end{tikzpicture}\]
In general, there will be zero or more black vertical arrows in between the green vertical arrow and the red vertical arrow. Each of the red arrows is crossed $k/n$ times, and suppose $a$ is crossed $s$ times. If the width of the rectangle is $t$ arrows, this means there are $s+t k/n$ paths $\mu_i$ entering the rectangle, and $(t+1)k/n$ paths exiting the rectangle. So we can conclude that $s=k/n$.  Therefore the collection of arrows $\mu_i$ cover the ladder diagram, omitting the arrow dual to $a_\emptyset$, exactly $k/n$ times.

 Finally, now suppose that $a_\emptyset$ is crossed $r$ times by the paths -- that is, $r$ of the $\mu_i$ are the path corresponding to the empty set. The remaining arrows do not cross $a_\emptyset$, but they cover the rest of the diagram $k/n$ times, so that means that there are $(n-1)k/n$ of them. Since $r+(n-1)k/n=k$, we conclude that $r=k/n$ as required. So the set $\{\mu_i\}$ is a $k/n$-covering. 
\end{proof}
This proposition characterizes a (redundant form of) the weight matrix. 
\begin{rem} We expect that the relations from $1$-coverings span all of the relations, but did not manage to prove this and do not require it in what follows. 
\end{rem}

\subsection{Grassmannian twins} \label{sec:twin}
One of the more interesting examples of Batyrev--Borisov mirror pairs obtained from lattice inclusions is the quintic twin~\cite{doran}. This is a hypersurface in a quotient of projective space whose Picard--Fuchs equations are identical to that of the quintic. The quotient is obtained by taking the action on~$\PP^4$ of the group $\mu_5$ of fifth roots of unity given by
\[\zeta \cdot [z_0:\cdots: z_4]=[z_0: \zeta z_1: \zeta^2 z_2: \zeta^3 z_3: \zeta^4 z_4]. \] 
Let $\tilde{P}$ denote the polytope of the quotient, a lattice polytope in $\tilde{M}$, and $P$ the polytope of projective space, a lattice polytope in $M$. Then the group action corresponds to a lattice inclusion $i:M \subset \tilde{M}$. The fact that the Picard--Fuchs equations are identical follows from the fact that $i^{-1}(\tilde{P})$ is isomorphic to $P$. 

The perspective on mirror symmetry for high-index Fano varieties described in this paper also suggests that the quintic twin can be generalized to Grassmannians, as in the following example.

\begin{eg} 
  Consider the action of the group $\mu_5$ on $\Mat(2 \times 5)$ given by multiplication on the right by diagonal matrices with entries $(1,\zeta,\dots, \zeta^4)$. This extends to an action on the ambient space of the Pl\U cker embedding, and hence also on the Gelfand--Cetlin degeneration.  Let $\tilde{P}_{2,5}$ denote the lattice polytope corresponding to the quotient by $K$ of the Gelfand--Cetlin degeneration. Denote the ambient lattice of $P_{2,5}$ as $M$, and the ambient lattice of $\tilde{P}_{2,5}$ as $\tilde{M}$.  As for projective space, there is an inclusion of lattices $i:M \subset \tilde{M}$, and $i^{-1}(\tilde{P}_{2,5}) \cong P_{2,5}$. The toric variety associated to $\tilde{P}_{2,5}$  is a toric degeneration of $\Gr(2,5)/K$. 
  
  Conjectures that form part of the Corti--Golyshev Fanosearch program~\cite{CoatesCortiGalkinGolyshevKasprzyk2013} would imply that the quantum periods of $\Gr(2,5)$ and $\Gr(2,5)/K$ -- which are generating functions for the genus-zero Gromov--Witten invariants of the two spaces -- can be computed as classical periods of Laurent polynomials $f$ and $\tilde{f}$ with Newton polytopes $P_{2,5}$ and $\tilde{P}_{2,5}$ respectively.  (The statement about $\Gr(2,5)$ has been proven by Marsh--Rietsch~\cite{MarshRietsch}.)  Since $f$ and $\tilde{f}$ differ only by changing the lattice of exponents, they have the same classical period. Thus we expect that the quantum periods of $\Gr(2,5)$ and $\Gr(2,5)/K$ should coincide. In other words, their anticanonical Calabi--Yau sections exhibit a `twin phenomenon'.
\end{eg}
 
\section{The Gelfand--Cetlin degeneration and variation of GIT} \label{sec:gt}
The Grassmannian degenerates to the toric variety $X_{P_{n,r}}$, to which Batyrev--Borisov duality can be applied. The primitive dual $Q_{n,r}$ of $P_{n,r}$ has been defined and constructed in the previous section. The Batyrev--Borisov mirror of $X_{P_{n,r}}$ is $X_{Q_{n,r}}/G$, and the  Batyrev--Borisov mirror of $X_{Q_{n,r}}$ is $X_{P_{n,r}}/G$, where $G \cong (\ZZ/n \ZZ)^{r(n-r)-1}.$ The action of the group is determined by the lattice inclusions that define the primitive dual. In this section, we will explore the relationship between $P_{n,r}$ and $Q_{n,r}$ in more detail. We will see that there is a weak Fano toric variety $Y_{n,r}$ that fits into the diagrams
 
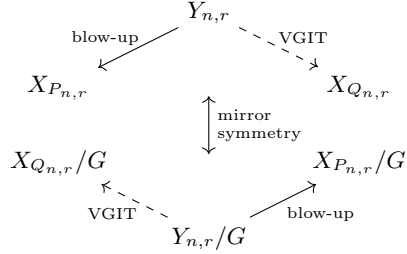
\begin{center}
  \begin{tikzpicture}\small
  \node at (0,0) (A) {$X_{P_{n,r}}$};
  \node at (4,0) (B) {$X_{Q_{n,r}}$};
  \node at (2,1) (C) {$Y_{n,r}$};
  \draw[dashed, ->](C)--(B)  node[align=right,pos=0.25,right] {\tiny \: VGIT};
  \draw[->](C)--(A)  node[pos=0.25,left] {\tiny blow-up \:};
  \node at (2,0) (s) {};
  \node at (2,-1) (t) {};
  \draw[<->](s)--(t) node[align=left,midway,right]{
    \begin{minipage}{1.2cm}
      \tiny mirror symmetry
    \end{minipage}};
  \node at (4,-1) (B1) {$X_{P_{n,r}}/G$};
  \node at (0,-1) (A1) {$X_{Q_{n,r}}/G$};
  \node at (2,-2) (C1) {$Y_{n,r}/G$};
  \draw[->](C1)--(B1)  node[pos=0.1,right] {\tiny \: \: blow-up};
  \draw[->,dashed](C1)--(A1)  node[align=left,pos=0.1,left] {\tiny VGIT \:\:};
   \end{tikzpicture}
\end{center}

\begin{rem}
  The VGITs here are $K$-equivalences, and so $X_{P_{n,r}}$ and $X_{Q_{n,r}}$ are related by (a natural generalisation of) a geometric transition. In the Calabi--Yau threefold case, if $X$ and $Y$ are related by a geometric transition then it is expected that their mirrors $X^\vee$ and $Y^\vee$ are also related by a geometric transition, with the roles of the birational contraction and the smoothing reversed. We see the same phenomenon here.
\end{rem}
 
Recall that the vertices of $P_{n,r}$ are indexed by arrows in the ladder quiver. The vertices of $Q_{n,r}$ are indexed by partitions $\lambda \in \Pnr.$ As before, set $k:=n-r$. We will need to consider a set of special partitions in $\Pnr$.  Define:
\begin{enumerate}
\item $\mu_{a,b}=(\underbrace{k,\dots,k}_{a \text{ times}},b)$ where $0 \leq a \leq r$ and $0 \leq b \leq k$.
\item $\nu_{a,b}$ to be the partition such that its transpose is $\nu_{a,b}^t=(\underbrace{r,\dots,r}_{a \text{ times}},b)$ where $0 \leq a \leq k$ and $0 \leq b \leq r$. 
\end{enumerate}
We will assign a partition $\mu_{a,b}$ or $\nu_{a,b}$ to each arrow in the dual quiver. This is equivalent to assigning such a partition to each arrow in the ladder quiver, since each arrow in the dual quiver crosses precisely one arrow in the ladder quiver. It is easiest to see how the arrows are assigned in an example -- see Figure \ref{figure:labeling}. 
\Yboxdim{3pt}
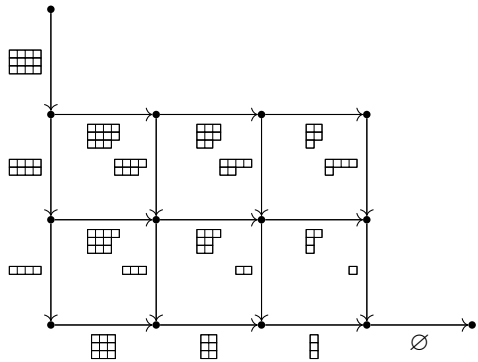
\begin{figure}[hb]
\caption{The labeling for $n=7$, $r=3$}
\[\begin{tikzpicture}[scale=1.4]
\node[circle,fill,scale=0.3] at (0.5,0.5) (a) {};
\node[circle,fill,scale=0.3] at (1.5,0.5) (b) {};
\node[circle,fill,scale=0.3] at (2.5,0.5) (c) {};
\node[circle,fill,scale=0.3] at (3.5,0.5) (d) {};
\node[circle,fill,scale=0.3] at (4.5,0.5) (l) {};

\node[circle,fill,scale=0.3] at (0.5,1.5) (e) {};
\node[circle,fill,scale=0.3] at (1.5,1.5) (i) {};
\node[circle,fill,scale=0.3] at (2.5,1.5) (k) {};
\node[circle,fill,scale=0.3] at (3.5,1.5) (m) {};

\node[circle,fill,scale=0.3] at (0.5,2.5) (n) {};
\node[circle,fill,scale=0.3] at (1.5,2.5) (o) {};
\node[circle,fill,scale=0.3] at (2.5,2.5) (p) {};
\node[circle,fill,scale=0.3] at (3.5,2.5) (q) {};

\node[circle,fill,scale=0.3] at (0.5,3.5) (f) {};
\draw[] (l) --(d)  node[midway,below] {\text{$\emptyset$}};
\draw[] (d)--(c)  node[midway,below] {\text{$\yng(1,1,1)$}};
\draw[] (c)--(b)  node[midway,below] {\text{$\yng(2,2,2)$}};
\draw[] (b)--(a) node[midway,below] {\text{$\yng(3,3,3)$}};

\draw[<-] (l) --(d);
\draw[<-] (d)--(c);
\draw[<-] (c)--(b);
\draw[<-] (b)--(a);

\draw[] (k)--(i) node[midway,below] {\yng(3,2,2)};
\draw[] (i)--(e) node[midway,below] {\yng(4,3,3)};
\draw[] (m)--(k) node[midway,below] {\yng(2,1,1)};

\draw[<-] (k)--(i);
\draw[<-] (i)--(e);
\draw[<-] (m)--(k);

\draw[] (o)--(n) node[midway,below] {\yng(4,4,3)};
\draw[] (p)--(o) node[midway,below] {\yng(3,3,2)};
\draw[] (q)--(p) node[midway,below] {\yng(2,2,1)};

\draw[<-] (o)--(n);
\draw[<-] (p)--(o);
\draw[<-] (q)--(p);

\draw[] (e)--(a)  node[midway,left] {\yng(4)};
\draw[] (n)--(e)  node[midway,left] {\yng(4,4)};
\draw[] (f)--(n)  node[midway,left] {\yng(4,4,4)};

\draw[->] (e)--(a);
\draw[->] (n)--(e);
\draw[->] (f)--(n);

\draw[] (i)--(b)   node[midway,left] {\yng(3)};
\draw[] (o)--(i)   node[midway,left] {\yng(4,3)};

\draw[->] (i)--(b);
\draw[->] (o)--(i);

\draw[->] (k)--(c);
\draw[->] (p)--(k);

\draw[] (k)--(c) node[midway,left] {\yng(2)};
\draw[] (p)--(k) node[midway,left] {\yng(4,2)};

\draw[->] (m)--(d);
\draw[->] (q)--(m);

\draw[] (m)--(d)  node[midway,left] {\yng(1)};
\draw[] (q)--(m)  node[midway,left] {\yng(4,1)};
\end{tikzpicture}\]
\label{figure:labeling}
\end{figure}

The following rules describe the assignment in general:
\begin{itemize}
\item To the arrow with source an external vertex, we assign the maximal partition $(k^r)$. To the arrow with target an external vertex, we assign $\emptyset$.
\item There are $r$ rows of horizontal arrows with internal vertices, with each row containing $k-1$ arrows. To the $a^{th}$ arrow (from the right) in the $b^{th}$ row (from the bottom) of horizontal arrows, we assign $\nu_{a,b-1}$.  
\item There are $k$ columns of vertical arrows, with each row containing $r-1$ arrows. To the $a^{th}$ arrow (from the bottom) in the $b^{th}$ column (from the right) of vertical arrows, we assign $\mu_{a-1,b}$.  
\end{itemize}
This defines a map from the arrows of the ladder quiver to $\Pnr$, which we call $\phi:\LQ_1 \to \Pnr$.

 \begin{mydef} A partition $\lambda \in \Pnr$ is an \emph{arrow partition} if it is in the image of $\phi: \LQ_1 \to \Pnr$, that is, if it is of the form:
\begin{enumerate}
\item $\nu_{a,b-1}$, for $1 \leq a \leq k$ and $1 \leq b \leq r-1$.
\item $\mu_{a-1,b}$ for $1 \leq a \leq r$ and  $1 \leq b \leq k-1$.
\end{enumerate}
The set of arrow partitions is denoted $\Anr$. Its complement in $\Pnr$ is labeled $\Bnr$; partitions in $\Bnr$ are called \emph{excess partitions}. 
\end{mydef}
There are $2(r-1)(n-r-1)+n$ arrow partitions.

\begin{eg} Let $r=3$,~$n=6$. Then 
  \[\Bnr=\{\yng(1,1),\yng(2,1),\yng(2,2),\yng(3,1,1),\yng(3,2,1),\yng(3,3,1)\}.\]
\end{eg}

\begin{lem}\label{lem:cd}
Consider a path in the dual quiver between the two external vertices. The set of partitions given by the labels of the arrows in the path are the set of partitions associated to some crossing diagram. 
\end{lem}
\begin{proof}
We say that two partitions are crossing diagram compatible if they can both appear in the same crossing diagram. Fix a dual arrow $a$, and consider all arrows in the dual quiver strictly to the right and below $a$ (i.e. all arrows that could appear after $a$ in a path on the dual quiver). Suppose $a$ is vertical, so that $\phi(a)$ is of the form
\[\begin{tikzpicture}[scale=0.4]
\draw (0,0) rectangle (5,5);
\draw[green,thick] (0,0)--(0,2)--(3,2)--(3,3)--(5,3)--(5,5);
\end{tikzpicture}\]
The two outgoing dual arrows from $t(a)$ have the labels drawn in orange and yellow below, and hence are each crossing diagram compatible with $\phi(a)$.
\[
\begin{tikzpicture}[scale=0.4]
\draw (0,0) rectangle (5,5);
\draw[green,thick] (0.05,0)--(0.05,2)--(3,2)--(3,3)--(5,3)--(5,5);
\draw[orange,thick] (0,0)--(0,3)--(3,3)--(3,4)--(4.9,4)--(4.9,5);

\end{tikzpicture}, \hspace{5mm} 
\begin{tikzpicture}[scale=0.4]
\draw (0,0) rectangle (5,5);
\draw[green,thick] (0,0)--(0,2)--(3,2)--(3,3)--(5,3)--(5,5);

\draw[yellow,thick] (0,0)--(2,0)--(2,3)--(3,3)--(3,5)--(5,5);
\end{tikzpicture}\]
 Since all other arrows below and to the right of $a$ have labels obtained by iteratively removing single boxes and complete rows or columns from these two, they are all also crossing-diagram compatible with $\phi(a)$. The same analysis shows the same statement for $a$ horizontal.  If we now take a path between the two external vertices of the dual quiver, it is clear that all of the labels along the path are contained in the same crossing diagram, and since there are $n$ of them, they are precisely the set of partitions associated to this crossing diagram. 
\end{proof}

The polytopes $P_{n,r}$ and $Q_{n,r}$ are closely related. Recall that the vertices of $Q_{n,r}$ are given by $v_\lambda$,~$\lambda \in \Pnr$.
\begin{thm} \label{thm:forgetvertices} Let $\overline{Q}_{n,r}$ be the polytope obtained by deleting the vertices of $Q_{n,r}$ associated to the excess partitions, i.e., $\overline{Q}_{n,r}$ is the convex hull of $\{v_\lambda : \lambda \in \Anr\}$. Then $P_{n,r}$ and $\overline{Q}_{n,r}$ are isomorphic as lattice polytopes.
\end{thm}
\begin{proof} To prove the statement, it suffices to show
\begin{enumerate}
\item that $\overline{Q}_{n,r}$ is a Fano polytope (that is, it has primitive vertices and the origin is in its strict interior);
\item that both polytopes have vertices that span the ambient lattice;
\item and that the associated toric varieties $X_{P_{n,r}}$ and $X_{\overline{Q}_{n,r}}$ have the same weight matrix. 
\end{enumerate}
We will first show a claim that will help with all three statements.

Since $\overline{Q}_{n,r}$ is obtained from $Q_{n,r}$ by forgetting some vertices, relations between vertices in the former are just given by relations between vertices in the latter that do not involve any of the forgotten vertices. We will describe a crossing diagram -- which in particular is a relation between the vertices of $Q_{n,r}$ -- for each non-source vertex in the ladder quiver, such that all associated paths are in $\Anr$. 

%Let $v$ be a non-source vertex in $\LQ$. Define a crossing diagram by assigning to $O$ to each vertex in the same row and to the right as $v$ and to each vertex in the same column above $v$. Assign $X$ to every other vertex. It easy to see that every partition of this crossing diagram is a partition in $\Anr$. In fact, we can be more precise. Fixing a vertex, it defines a row in the weight matrix of $P_{n,r}$. This weight matrix consists of $0$s and $1$s, and the locations of the $1$s in the row of a particular vertex (i.e. $a$ such that $b_{va}=1$) have been identified, see Example \ref{eg:1set}.
\begin{claim} \label{claim:crossing}
  Fix a non-source vertex $v$. Then the set $\{\phi(a): b_{va}=1\}$ is precisely the set of paths for some crossing diagram. 
\end{claim}
\begin{proof}[Proof of Claim]
We have described the set $\{\phi(a): b_{va}=1\}$ in Example \ref{eg:1set}. It will be more convenient if we dualize the picture given there: the highlighted vertex  (if internal) becomes a square in the dual quiver, and the highlighted arrows correspond to highlighted dual arrows. In the case of the external vertex, no square is highlighted.  The diagrams then become:
\[\begin{tikzpicture}[scale=0.6]
\draw[fill, yellow] (0.5,0.5) rectangle (1.5,1.5);
\node[circle,fill,scale=0.3] at (0.5,0.5) (a) {};
\node[circle,fill,scale=0.3] at (1.5,0.5) (b) {};
\node[circle,fill,scale=0.3] at (2.5,0.5) (c) {};
\node[circle,fill,scale=0.3] at (3.5,0.5) (d) {};
\node[circle,fill,scale=0.3] at (4.5,0.5) (l) {};

\node[circle,fill,scale=0.3] at (0.5,1.5) (e) {};
\node[circle,fill,scale=0.3] at (1.5,1.5) (i) {};
\node[circle,fill,scale=0.3] at (2.5,1.5) (k) {};
\node[circle,fill,scale=0.3] at (3.5,1.5) (m) {};

\node[circle,fill,scale=0.3] at (0.5,2.5) (n) {};
\node[circle,fill,scale=0.3] at (1.5,2.5) (o) {};
\node[circle,fill,scale=0.3] at (2.5,2.5) (p) {};
\node[circle,fill,scale=0.3] at (3.5,2.5) (q) {};

\node[circle,fill,scale=0.3] at (0.5,3.5) (f) {};

\draw[red,<-] (l) --(d);
\draw[red,<-] (d)--(c);
\draw[red, <-] (c)--(b);
\draw[red,<-] (b)--(a);

\draw[<-] (k)--(i);
\draw[<-] (i)--(e);
\draw[<-] (m)--(k);

\draw[<-] (o)--(n);
\draw[<-] (p)--(o);
\draw[<-] (q)--(p);

\draw[red, ->] (e)--(a);
\draw[red,->] (n)--(e);
\draw[red, ->] (f)--(n);

\draw[->] (i)--(b);
\draw[->] (o)--(i);

\draw[->] (k)--(c);
\draw[->] (p)--(k);

\draw[->] (m)--(d);
\draw[->] (q)--(m);
\end{tikzpicture}, \hspace{5mm} \begin{tikzpicture}[scale=0.6]
\draw[fill, yellow] (1.5,0.5) rectangle (2.5,1.5);
\node[circle,fill,scale=0.3] at (0.5,0.5) (a) {};
\node[circle,fill,scale=0.3] at (1.5,0.5) (b) {};
\node[circle,fill,scale=0.3] at (2.5,0.5) (c) {};
\node[circle,fill,scale=0.3] at (3.5,0.5) (d) {};
\node[circle,fill,scale=0.3] at (4.5,0.5) (l) {};

\node[circle,fill,scale=0.3] at (0.5,1.5) (e) {};
\node[circle,fill,scale=0.3] at (1.5,1.5) (i) {};
\node[circle,fill,scale=0.3] at (2.5,1.5) (k) {};
\node[circle,fill,scale=0.3] at (3.5,1.5) (m) {};

\node[circle,fill,scale=0.3] at (0.5,2.5) (n) {};
\node[circle,fill,scale=0.3] at (1.5,2.5) (o) {};
\node[circle,fill,scale=0.3] at (2.5,2.5) (p) {};
\node[circle,fill,scale=0.3] at (3.5,2.5) (q) {};

\node[circle,fill,scale=0.3] at (0.5,3.5) (f) {};

\draw[red,<-] (l) --(d);
\draw[red,<-] (d)--(c);
\draw[red,<-] (c)--(b);
\draw[<-] (b)--(a);

\draw[<-] (k)--(i);
\draw[<-] (i)--(e);
\draw[<-] (m)--(k);

\draw[red, <-] (o)--(n);
\draw[<-] (p)--(o);
\draw[<-] (q)--(p);

\draw[->] (e)--(a);
\draw[->] (n)--(e);
\draw[red,->] (f)--(n);

\draw[red,->] (i)--(b);
\draw[red,->] (o)--(i);

\draw[->] (k)--(c);
\draw[->] (p)--(k);

\draw[->] (m)--(d);
\draw[->] (q)--(m);
\end{tikzpicture},  \hspace{5mm}
 \begin{tikzpicture}[scale=0.6]
 \draw[fill, yellow] (2.5,0.5) rectangle (3.5,1.5);
\node[circle,fill,scale=0.3] at (0.5,0.5) (a) {};
\node[circle,fill,scale=0.3] at (1.5,0.5) (b) {};
\node[circle,fill,scale=0.3] at (2.5,0.5) (c) {};
\node[circle,fill,scale=0.3] at (3.5,0.5) (d) {};
\node[circle,fill,scale=0.3] at (4.5,0.5) (l) {};

\node[circle,fill,scale=0.3] at (0.5,1.5) (e) {};
\node[circle,fill,scale=0.3] at (1.5,1.5) (i) {};
\node[circle,fill,scale=0.3] at (2.5,1.5) (k) {};
\node[circle,fill,scale=0.3] at (3.5,1.5) (m) {};

\node[circle,fill,scale=0.3] at (0.5,2.5) (n) {};
\node[circle,fill,scale=0.3] at (1.5,2.5) (o) {};
\node[circle,fill,scale=0.3] at (2.5,2.5) (p) {};
\node[circle,fill,scale=0.3] at (3.5,2.5) (q) {};

\node[circle,fill,scale=0.3] at (0.5,3.5) (f) {};

\draw[red,<-] (l) --(d);
\draw[red,<-] (d)--(c);
\draw[<-] (c)--(b);
\draw[<-] (b)--(a);

\draw[<-] (k)--(i);
\draw[<-] (i)--(e);
\draw[<-] (m)--(k);

\draw[red,<-] (o)--(n);
\draw[red,<-] (p)--(o);
\draw[<-] (q)--(p);

\draw[->] (e)--(a);
\draw[->] (n)--(e);
\draw[red,->] (f)--(n);

\draw[->] (i)--(b);
\draw[->] (o)--(i);

\draw[red,->] (k)--(c);
\draw[red,->] (p)--(k);

\draw[->] (m)--(d);
\draw[->] (q)--(m);
\end{tikzpicture},\]
\[\begin{tikzpicture}[scale=0.6]
\draw[fill, yellow] (0.5,1.5) rectangle (1.5,2.5);
\node[circle,fill,scale=0.3] at (0.5,0.5) (a) {};
\node[circle,fill,scale=0.3] at (1.5,0.5) (b) {};
\node[circle,fill,scale=0.3] at (2.5,0.5) (c) {};
\node[circle,fill,scale=0.3] at (3.5,0.5) (d) {};
\node[circle,fill,scale=0.3] at (4.5,0.5) (l) {};

\node[circle,fill,scale=0.3] at (0.5,1.5) (e) {};
\node[circle,fill,scale=0.3] at (1.5,1.5) (i) {};
\node[circle,fill,scale=0.3] at (2.5,1.5) (k) {};
\node[circle,fill,scale=0.3] at (3.5,1.5) (m) {};

\node[circle,fill,scale=0.3] at (0.5,2.5) (n) {};
\node[circle,fill,scale=0.3] at (1.5,2.5) (o) {};
\node[circle,fill,scale=0.3] at (2.5,2.5) (p) {};
\node[circle,fill,scale=0.3] at (3.5,2.5) (q) {};

\node[circle,fill,scale=0.3] at (0.5,3.5) (f) {};

\draw[red,<-] (l) --(d);
\draw[<-] (d)--(c);
\draw[<-] (c)--(b);
\draw[<-] (b)--(a);

\draw[red,<-] (k)--(i);
\draw[red,<-] (i)--(e);
\draw[red,<-] (m)--(k);

\draw[<-] (o)--(n);
\draw[<-] (p)--(o);
\draw[<-] (q)--(p);

\draw[->] (e)--(a);
\draw[red,->] (n)--(e);
\draw[red,->] (f)--(n);

\draw[->] (i)--(b);
\draw[->] (o)--(i);

\draw[->] (k)--(c);
\draw[->] (p)--(k);

\draw[red,->] (m)--(d);
\draw[->] (q)--(m);
\end{tikzpicture}, \hspace{5mm} \begin{tikzpicture}[scale=0.6]
\draw[fill, yellow] (1.5,1.5) rectangle (2.5,2.5);
\node[circle,fill,scale=0.3] at (0.5,0.5) (a) {};
\node[circle,fill,scale=0.3] at (1.5,0.5) (b) {};
\node[circle,fill,scale=0.3] at (2.5,0.5) (c) {};
\node[circle,fill,scale=0.3] at (3.5,0.5) (d) {};
\node[circle,fill,scale=0.3] at (4.5,0.5) (l) {};

\node[circle,fill,scale=0.3] at (0.5,1.5) (e) {};
\node[circle,fill,scale=0.3] at (1.5,1.5) (i) {};
\node[circle,fill,scale=0.3] at (2.5,1.5) (k) {};
\node[circle,fill,scale=0.3] at (3.5,1.5) (m) {};

\node[circle,fill,scale=0.3] at (0.5,2.5) (n) {};
\node[circle,fill,scale=0.3] at (1.5,2.5) (o) {};
\node[circle,fill,scale=0.3] at (2.5,2.5) (p) {};
\node[circle,fill,scale=0.3] at (3.5,2.5) (q) {};

\node[circle,fill,scale=0.3] at (0.5,3.5) (f) {};

\draw[red,<-] (l) --(d);
\draw[<-] (d)--(c);
\draw[<-] (c)--(b);
\draw[<-] (b)--(a);

\draw[red,<-] (k)--(i);
\draw[<-] (i)--(e);
\draw[red,<-] (m)--(k);

\draw[red,<-] (o)--(n);
\draw[<-] (p)--(o);
\draw[<-] (q)--(p);

\draw[->] (e)--(a);
\draw[->] (n)--(e);
\draw[red,->] (f)--(n);

\draw[->] (i)--(b);
\draw[red,->] (o)--(i);

\draw[->] (k)--(c);
\draw[->] (p)--(k);

\draw[red,->] (m)--(d);
\draw[->] (q)--(m);
\end{tikzpicture},  \hspace{5mm}
 \begin{tikzpicture}[scale=0.6]
 \draw[fill, yellow] (2.5,1.5) rectangle (3.5,2.5);
\node[circle,fill,scale=0.3] at (0.5,0.5) (a) {};
\node[circle,fill,scale=0.3] at (1.5,0.5) (b) {};
\node[circle,fill,scale=0.3] at (2.5,0.5) (c) {};
\node[circle,fill,scale=0.3] at (3.5,0.5) (d) {};
\node[circle,fill,scale=0.3] at (4.5,0.5) (l) {};

\node[circle,fill,scale=0.3] at (0.5,1.5) (e) {};
\node[circle,fill,scale=0.3] at (1.5,1.5) (i) {};
\node[circle,fill,scale=0.3] at (2.5,1.5) (k) {};
\node[circle,fill,scale=0.3] at (3.5,1.5) (m) {};

\node[circle,fill,scale=0.3] at (0.5,2.5) (n) {};
\node[circle,fill,scale=0.3] at (1.5,2.5) (o) {};
\node[circle,fill,scale=0.3] at (2.5,2.5) (p) {};
\node[circle,fill,scale=0.3] at (3.5,2.5) (q) {};

\node[circle,fill,scale=0.3] at (0.5,3.5) (f) {};

\draw[red,<-] (l) --(d);
\draw[<-] (d)--(c);
\draw[<-] (c)--(b);
\draw[<-] (b)--(a);

\draw[<-] (k)--(i);
\draw[<-] (i)--(e);
\draw[red,<-] (m)--(k);

\draw[red,<-] (o)--(n);
\draw[red,<-] (p)--(o);
\draw[<-] (q)--(p);

\draw[->] (e)--(a);
\draw[->] (n)--(e);
\draw[red,->] (f)--(n);

\draw[->] (i)--(b);
\draw[->] (o)--(i);

\draw[->] (k)--(c);
\draw[red,->] (p)--(k);

\draw[red,->] (m)--(d);
\draw[->] (q)--(m);
\end{tikzpicture}, 
 \hspace{5mm} \begin{tikzpicture}[scale=0.6]
\node[circle,fill,scale=0.3] at (0.5,0.5) (a) {};
\node[circle,fill,scale=0.3] at (1.5,0.5) (b) {};
\node[circle,fill,scale=0.3] at (2.5,0.5) (c) {};
\node[circle,fill,scale=0.3] at (3.5,0.5) (d) {};
\node[circle,fill,scale=0.3] at (4.5,0.5) (l) {};

\node[circle,fill,scale=0.3] at (0.5,1.5) (e) {};
\node[circle,fill,scale=0.3] at (1.5,1.5) (i) {};
\node[circle,fill,scale=0.3] at (2.5,1.5) (k) {};
\node[circle,fill,scale=0.3] at (3.5,1.5) (m) {};

\node[circle,fill,scale=0.3] at (0.5,2.5) (n) {};
\node[circle,fill,scale=0.3] at (1.5,2.5) (o) {};
\node[circle,fill,scale=0.3] at (2.5,2.5) (p) {};
\node[circle,fill,scale=0.3] at (3.5,2.5) (q) {};

\node[circle,fill,scale=0.3] at (0.5,3.5) (f) {};

\draw[red,<-] (l) --(d);
\draw[<-] (d)--(c);
\draw[<-] (c)--(b);
\draw[<-] (b)--(a);

\draw[<-] (k)--(i);
\draw[<-] (i)--(e);
\draw[<-] (m)--(k);

\draw[red,<-] (o)--(n);
\draw[red,<-] (p)--(o);
\draw[red,<-] (q)--(p);

\draw[->] (e)--(a);
\draw[->] (n)--(e);
\draw[red,->] (f)--(n);

\draw[->] (i)--(b);
\draw[->] (o)--(i);

\draw[->] (k)--(c);
\draw[->] (p)--(k);

\draw[red,->] (m)--(d);
\draw[red,->] (q)--(m);
\end{tikzpicture}\]
and this generalises easily to arbitrary $n$ and $r$. The important thing to see from this picture is that the collection of dual arrows satisfying $b_{va}=1$ forms a path between the external vertices of the dual quiver. Therefore by Lemma \ref{lem:cd}, the set of labels are the partitions for a crossing diagram. 
\end{proof}

Claim~\ref{claim:crossing} implies that the weight matrix of $X_{\overline{Q}_{n,r}}$ contains all of the rows of the weight matrix $X_{P_{n,r}}$, although there may be additional rows (such as those covering from $k$-coverings for $k \geq 2$). 

We need now to show that $\overline{Q}_{n,r}$ is a Fano polytope with vertices that span the lattice: we already know that this holds for $P_{n,r}$. Note that the vertices of $\overline{Q}_{n,r}$ are primitive lattice vectors, as they are also vertices of the reflexive polytope $Q_{n,r}$. For every vertex $v_\lambda$ of $\overline{Q}_{n,r}$, there is at least one crossing diagram of the form in Claim~\ref{claim:crossing} which has $\lambda$ as a path, and this implies that there exist positive $c_\lambda$ such that 
\[
    0=\sum_{\lambda \in \Anr} c_\lambda v_\lambda
\]
Thus $0$ lies in the strict interior of $\overline{Q}_{n,r}$. To see that the vertices span the lattice, note that it suffices to check that the $m_\lambda$,~$\lambda \in \Anr$ span the lattice $N$: this follows from the explicit description of the $m_\lambda$ in \eqref{eqn:mlambda}. We leave the details here to the reader. It follows that $\overline{Q}_{n,r}$ is an $r(n-r)$-dimensional Fano polytope, with the same number of vertices as $P_{n,r}$. Therefore there can be no more relations between the vertices then those already found -- we have already found $V - r(n-r)$ relations, where $V$ is the number of vertices of $\overline{Q}_{n,r}$, and we know these relations are linearly independent because they are also relations for $P_{n,r}$ -- so the weight matrices are the same. 
\end{proof}
Since we are only interested in studying $P_{n,r}$ up to isomorphism, we can replace $P_{n,r}$ with $\overline{Q}_{n,r}$, and view $P_{n,r}$ as a polytope in $N$ with vertices indexed by $\lambda \in \Anr$.

Theorem~\ref{thm:forgetvertices} tells us that $Q_{n,r}$ is obtained from $P_{n,r}$ by adding vertices. This has the effect of blowing up the spanning fan $F_{P_{n,r}}$ of $P_{n,r}$ along the missing rays (in some order). Label the partitions in $\Bnr$ as $\lambda_1,\dots,\lambda_m$, and set $v_i:=v_{\lambda_i}\in N$.  Blowing up the spanning fan of $P_{n,r}$ (viewed as a polytope in $N_\RR$) repeatedly at $v_1,\dots,v_m$, we obtain a new fan, and hence a new toric variety, which was called $Y_{n,r}$ in the introduction. 

\begin{thm}\label{thm:vgit}There is a variation of GIT from the weak Fano variety $Y_{n,r}$ to $X_{Q_{n,r}}$. The anticanonical polytope of these two toric varieties are equal, and they have isomorphic Cox rings. In particular, $\Gamma(\mathcal{O}_{-K_{Y_{n,r}}})=\Gamma(\mathcal{O}_{-K_{X_{Q_{n,r}}}}).$
\end{thm}
\begin{proof}
It follows from Theorem \ref{thm:forgetvertices} that the GIT description of $Y_{n,r}$ and $X_{Q_{n,r}}$ has the same weight matrix -- i.e.~the fans have the same rays. The difference lies entirely in the choice of stability condition, which gives the different fan structure. The dual of the anticanonical polytope is the polytope spanned by primitive generators of the rays of the fan, so the anticanonical polytope of both toric varieties agree.  In particular since this polyhedron is convex, we have that the anticanonical bundle of $Y_{n,r}$ is nef; that is, $Y_{n,r}$ is weak Fano. The Cox ring of a toric variety depends only on the rays of the fan, so the last two claims follow immediately.
\end{proof}
\section{Smoothing} \label{sec:smoothing}
In this section, we will discuss (partial) smoothings of $X_{P_{n,r}}/G$ and of $Y_{n,r}/G$.
\subsection{Smoothing $X_{P_{n,r}}/G$}
Although the group action of $G$ on $X_{P_{n,r}}$ looks very toric and unconnected to the Grassmannian, remarkably, when smoothing, we obtain a very natural group action on $\Gr(n,r)$, which generalizes the Greene--Plesser  group action for projective space.  Recall that
\begin{equation} \tilde{H}_{n,r}=\langle (\zeta_1,\dots,\zeta_{n}) \mid \zeta_i^n=1, \, (\prod \zeta_i)^r=1 \rangle\end{equation}
acts naturally on $\Mat(r \times n)$ and descends to an effective action of $H_{n,r}$ on $\Gr(n,r)$. The group $H_{n,r}$ is the quotient
\[\tilde{H}_{n,r}/((\zeta,\dots,\zeta) \in \tilde{H}_{n,r}).\]

\subsubsection{The action of $G$ on $X_{P_{n,r}}$}
It follows from Lemma \ref{lem:Gdes} that $G \cong (\ZZ/n\ZZ)^{rk-1}$, where as before, $k=n-r$. The group action of $G$ on $X_{P_{n,r}}$ can be described as a quotient of a larger group action on $\CC^{|\LQ_1|}$, recalling that $X_{P_{n,r}}=\CC^{|LQ_1|}/\!/ T$, $T=(\CC^*)^{|\overline{\LQ}_0|}$. The group 
\[\tilde{G}:=\{ (\psi_a)_{a \in \LQ_1} \mid \psi_a^n=1, \, \prod_{a} \psi_a =1 \}\]
acts naturally via $\SL(|\LQ_1|)$ on $\CC^{|\LQ_1|}$. The action descends to give an effective action of the quotient
\[ G=\tilde{G}/ (T \cap \tilde{G})\]
on the toric variety $X_{P_{n,r}}$. To compute the intersection, we view both $T$ and $\tilde{G}$ as diagonal subgroups of $\GL(|\LQ_1|)$. Recall from \eqref{eq:weightmatrix} that the weight matrix of the toric variety is $B:=[b_{va}]_{v \in \overline{\LQ}_0, a \in \LQ_1}$.  The intersection $T \cap \tilde{G}$
is generated by
\begin{equation}\label{eqn:generators} \{(\psi^{b_{va}})_a \mid \psi^n=1, v \in \overline{\LQ}_0\} \subset \tilde{G}.\end{equation}

\begin{rem} One way to see that this group action agrees with the action described in Lemma \ref{lem:Gdes} is to convert the fan description of $X_{Q_{n,r}^\vee}$ to a GIT description. Recall that the vertices of $X_{Q_{n,r}^\vee}$ are indexed by arrows of the ladder diagram, and that $Q_{n,r}^\vee$ can be obtained from the polytope of a divisor of $X_{Q_{n,r}}$, namely the divisor associated to the empty partition.  Let $m_a$,~$a \in \LQ_1$, denote the vertices of the polytope of this divisor. Then the vertices of $Q_{n,r}^\vee$ are given by $w_a:=n m_a-h$, where $h$ is a (fixed) primitive lattice vector and $a \in \LQ_1$. 

The rays of the fan of $X_{Q_{n,r}^\vee}$ do not generate the ambient lattice. To convert to the GIT description, we must add virtual rays -- rays that appear in no cones of the fan -- so that all the rays together generate the lattice. It suffices to add all of the $m_a$ as virtual rays. Then in addition to the weights of $X_{P_{n,r}}$, we also obtain gradings coming from the relations 
\[w_a+h-n m_a=0.\]
Choosing an appropriate stability condition allows us to translate these gradings into \emph{quotient gradings}, which exactly give the action of $\tilde{G}$ on $\CC^{|\LQ_1|}$. 
\end{rem}
\subsubsection{The action of $G$ on $\PP^{\binom{n}{r}-1}$}
The embedding of $X_{P_{n,r}}$ into  $\PP^{\binom{n}{r}-1}$ allows us to extend the action of $G$ to  $\PP^{\binom{n}{r}-1}$. Label the coordinates of  $\PP^{\binom{n}{r}-1}$ by $p_\lambda$,~$\lambda \in \Pnr$. These are of course the Pl\U cker coordinates -- recall that they can be indexed either by partitions in $\Pnr$ or by size-$r$ subsets of $\{1,\dots,n\}$ (Definition \ref{def:indexing}). For a partition $\lambda \in \Pnr$, $I_\lambda$ denotes the corresponding subset. 

 The embedding is given by the $\binom{n}{r}$ sections of $\mathcal{O}(1)$ on $X_{P_{n,r}}$, i.e.
\[\prod_{a \in \lambda} y_a, \: \lambda \in \Pnr.\]
The group $\tilde{G}$  acts on the coordinate $p_\lambda$ by multiplication by $\prod_{a \in \lambda} \psi_a.$ Either by observing that it is true by construction of the embedding, or by explicitly checking, it is easy to see that the action of $\tilde{G}$ descends to an action of $G$ on  $\PP^{\binom{n}{r}-1}$. That is, elements of intersection $T \cap \tilde{G}$ act by scaling.

\subsubsection{Degenerate Pl\U cker relations} The toric variety $X_{P_{n,r}}$ is cut out of  $\PP^{\binom{n}{r}-1}$ by degenerate Pl\U cker relations \cite{flagdegenerations}.
\begin{mydef} Let $\prec$ be the usual partial order on partitions/Pl\U cker coordinates:
for $\lambda$,~$\sigma \in \Pnr$,  $\lambda \prec \sigma$ if $\lambda_i \leq \sigma_i$ for all $i$. Here we use zeroes to extend any partition in $\Pnr$ to a length-$r$ partition. 
\end{mydef}
Following \cite{flagdegenerations}, there is a relation for every incomparable pair of partitions $\lambda$,~$\sigma$:
\[p_\sigma p_\lambda -p_{\sigma \wedge \lambda} p_{\sigma \vee \lambda}.\]
Here $\sigma \wedge \lambda$ is the minimal partition $\mu$ satisfying $\lambda, \sigma \prec \mu$ and $\sigma \vee \lambda$ is the maximal partition $\nu$ satisfying $\nu \prec \lambda,\sigma.$
\begin{eg}
 We give an example for $Gr(7,3)$. The partitions $\sigma=(1,1,1)$ and $\lambda=(4,2)$ are incomparable. We draw them in blue and green respectively below.  
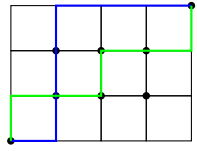
\[\begin{tikzpicture}[scale=0.6]
\draw (0,0) rectangle (1,1);
\draw (0,1) rectangle (1,2);
\draw (0,2) rectangle (1,3);
\draw[] (1,1) rectangle (2,2);
\draw[] (1,0) rectangle (2,1);
\draw[] (1,2) rectangle (2,3);
\draw[] (2,1) rectangle (3,2);
\draw[] (2,2) rectangle (3,3);
\draw[] (2,0) rectangle (3,1);
\draw[] (3,0) rectangle (4,1);
\draw[] (3,1) rectangle (4,2);
\draw[] (3,2) rectangle (4,3);
\draw[fill] (0,0) circle (2pt);
\draw[fill] (1,1) circle (2pt);
\draw[fill] (2,1) circle (2pt);
\draw[fill] (3,1) circle (2pt);
\draw[fill] (1,2) circle (2pt);
\draw[fill] (2,2) circle (2pt);
\draw[fill] (3,2) circle (2pt);
\draw[fill] (4,3) circle (2pt);

\draw[blue,thick] (0,0)--(1,0)--(1,1)--(1,2)--(1,3)--(4,3);
\draw[green,thick] (0,0)--(0,1)--(2,1)--(2,2)--(4,2)--(4,3);
\end{tikzpicture}\]
The partition $\sigma \vee \lambda$ is drawn in red and $\sigma \wedge \lambda$ in yellow. 
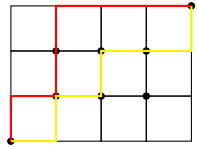
\[\begin{tikzpicture}[scale=0.6]
\draw (0,0) rectangle (1,1);
\draw (0,1) rectangle (1,2);
\draw (0,2) rectangle (1,3);
\draw[] (1,1) rectangle (2,2);
\draw[] (1,0) rectangle (2,1);
\draw[] (1,2) rectangle (2,3);
\draw[] (2,1) rectangle (3,2);
\draw[] (2,2) rectangle (3,3);
\draw[] (2,0) rectangle (3,1);
\draw[] (3,0) rectangle (4,1);
\draw[] (3,1) rectangle (4,2);
\draw[] (3,2) rectangle (4,3);
\draw[fill] (0,0) circle (2pt);
\draw[fill] (1,1) circle (2pt);
\draw[fill] (2,1) circle (2pt);
\draw[fill] (3,1) circle (2pt);
\draw[fill] (1,2) circle (2pt);
\draw[fill] (2,2) circle (2pt);
\draw[fill] (3,2) circle (2pt);
\draw[fill] (4,3) circle (2pt);

\draw[red,thick] (0,0)--(0,1)--(1,1)--(1,2)--(1,3)--(4,3);
\draw[yellow,thick]  (0,0)--(1,0)--(1,1)--(2,1)--(2,2)--(4,2)--(4,3);
\end{tikzpicture}\]
\end{eg}
\begin{lem} The degenerate Pl\U cker relations are homogeneous with respect to the induced action of $G$ on $\PP^{\binom{n}{r}-1}$. 
\end{lem}
\begin{proof} Let $\sigma$,~$\lambda$ be an incomparable pair. The set of arrows contained in the union of $\sigma$ and $\lambda$ is the same as the set of arrows contained in the union of $\sigma \vee \lambda$ and $\sigma \wedge \lambda$. Therefore $G$ acts with the same weights on $p_\sigma p_\lambda$ and $p_{\sigma \wedge \lambda} p_{\sigma \vee \lambda}.$
\end{proof}
\subsubsection{The group preserving the Pl\U cker relations}
Consider the Pl\U cker relation for $\Gr(4,2)$:
\[p_{\emptyset} p_{\yng(2,2)} -p_{\yng(1)} p_{\yng(2,1)} +p_{\yng(2)} p_{\yng(1,1)},\]
which degenerates to
\[0 \cdot p_{\emptyset} p_{\yng(2,2)} -p_{\yng(1)} p_{\yng(2,1)} +p_{\yng(2)} p_{\yng(1,1)}.\]
The second equation is homogeneous under the action of $G$, but the first is not. However, it is homogeneous with respect to a subgroup of $G$.
\begin{mydef} Let $G_h$ be the subgroup of $G$ that preserves the Grassmannian, i.e. $G_h$ is the largest subgroup of $G$ such that all Pl\U cker relations are homogeneous with respect to this action. 
\end{mydef}
Before describing $G_h$ precisely, let's describe it heuristically. Binomial Pl\U cker relations are of the form:
\[p_\sigma p_\lambda+\sum a_{\alpha,\beta} p_\alpha p_\beta =0.\]
For a complete description of these relations, see \cite{combinatorial}. A term $p_\alpha p_\beta$ appears with non-zero coefficient only if  the multi-sets $I_\sigma \cup I_\lambda$ and $I_\alpha \cup I_\beta$ agree. In other words, the terms that appear in the sum must just be re-arrangements of the indices of $I_\sigma$ and $I_\lambda$. The Pl\U cker relations are not homogeneous with respect to the action of $G$ because this action feels the paths of the partitions, which is more precise than the location of the vertical steps (which gives the underlying subset). We are looking for a subgroup $G_h$ whose action depends only on the indices. We will show that $G_h$ is isomorphic to $H_{n,r}$. The first step is to define a map from $H_{n,r}$ to $G$. 

Consider the following labeling of the vertical steps of a ladder diagram: 

 
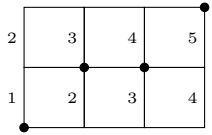
\[\begin{tikzpicture}[scale=0.8]
\draw (0,0) rectangle (1,1);
\draw (0,1) rectangle (1,2);
\draw (1,1) rectangle (2,2);

\draw (1,0) rectangle (2,1);
\draw (2,1) rectangle (3,2);

\draw (2,0) rectangle (3,1);
\draw[fill] (0,0) circle (2pt);
\draw[fill] (1,1) circle (2pt);
\draw[fill] (2,1) circle (2pt);
\draw[fill] (3,2) circle (2pt);

\node at (-0.2,0.5) (A) {\tiny 1};
\node at (-0.2,1.5) (A1) { \tiny 2};
\node at (0.8,0.5) (B) { \tiny 2};
\node at (0.8,1.5) (B1) { \tiny 3};
\node at (1.8,0.5) (C) { \tiny 3};
\node at (1.8,1.5) (C1) { \tiny 4};
\node at (2.8,0.5) (D) { \tiny 4};
\node at (2.8,1.5) (D1) { \tiny 5};
\end{tikzpicture}\]
\Yboxdim{6pt}
The labeling allows to assign a subset of $\{1,\dots,n\}$ to an arrow or path in the ladder diagram by recording the labels along the arrow or path. For example, the path corresponding to the partition $\yng(1)$ in the $\Gr(2,5)$ diagram has labeling $\{1,3\}$. Some arrows are associated to the empty set -- in this example, the internal horizontal arrow has no labels. 
\Yboxdim{3pt}

The following remark is key to establishing the next two lemmas. 
\begin{rem}The labeling of a path $\lambda$ is $I_\lambda$.  
\end{rem}

Use the labeling to define a morphism 
\[\tilde{\Psi}: \tilde{H}_{n,r} \to \tilde{G}, \:\: (\zeta_i)_{i=1}^n \mapsto \left(\prod_{\substack{\text{$j$ a label} \\ \text{of $a$}}} \zeta_j\right)_{a \in \LQ_1}.\]
\begin{lem}The map $\tilde{\Psi}$ descends to an injective map $\psi: H_{n,r} \to G.$
\end{lem}
\begin{proof} Recall that $G$ acts on $\PP^{\binom{n}{r}-1}$.  We know that $T \cap \tilde{G}$ acts on $\PP^{\binom{n}{r}-1}$ by scaling, so it suffices to show that for $h \in \tilde{H}_{n,r}$,
\[h=(\zeta,\dots,\zeta) \Leftrightarrow \tilde{\Psi}(h) \text{ acts by scaling on } \PP^{\binom{n}{r}-1}.\]
Note that 
\[\tilde{\Psi}((\zeta_i)_{i=1}^n)\cdot \left[p_\lambda: \lambda \in \Pnr\right]=\left[\left(\prod_{j \in I_\lambda} \zeta_j\right) p_\lambda: \lambda \in \Pnr\right].\]
Since $I_\lambda$ runs over all size $r$ subsets of $\{1,\dots,n\}$, the only way to ensure that 
\[\prod_{j \in I_\lambda} \zeta_j =\prod_{j \in I_\mu} \zeta_j \]
for all $\lambda$,~$\mu$ is for $\zeta_i=\zeta_j$ for all $i$,~$j$. 
\end{proof}

\begin{lem}\label{lem:imagepreserves} The image of $H_{n,r}$ under $\Psi$ preserves the Grassmannian, that is, 
\[ \Psi(H_{n,r}) \subseteq G_h.\]
\end{lem}
\begin{proof}
Consider a binomial Pl\U cker relation:
\[p_\sigma p_\lambda+\sum a_{\alpha,\beta} p_\alpha p_\beta =0,\]
where the term $p_\alpha p_\beta$ appears with non-zero coefficient only if  the multi-sets $I_\sigma \cup I_\lambda$ and $I_\alpha \cup I_\beta$ are equal. A group element $\tilde{\Psi}((\zeta_i))$ acts on  $p_\alpha p_\beta$ by multiplication by
\[\prod_{j \in I_\alpha} \zeta_j \prod_{j \in I_\beta} \zeta_j \]
which is equal to 
\[\prod_{j \in I_\lambda} \zeta_j \prod_{j \in I_\sigma} \zeta_j.\]
Therefore the relation is homogeneous with respect to the $\tilde{\Psi}(\tilde{G})$ action, and hence the $\Psi(G)$ action. 
\end{proof}
The main combinatorial result of this section is the next theorem.
\begin{thm}\label{thm:groupiso} There is an isomorphism $G_h \cong H_{n,r}$. 
\end{thm}
\begin{proof}
The map $\Psi$ defines an isomorphism $H_{n,r} \to \Psi(H_{n,r})$. We will show that $\Psi(H_{n,r})=G_h$; by Lemma \ref{lem:imagepreserves}, $\Psi(H_{n,r}) \subseteq G_h$. We will show the other direction by demonstrating that for every $[g] \in G_h$ with lift $g \in \tilde{G}$, the $\tilde{G} \cap T$-orbit of $g$ has non-empty intersection with $\tilde{\Psi}(\tilde{H}_{n,r})$. We will use two collections of arrows in the proof. The first collection $C_1$ is described via the following example: it is the set of $n=9$ arrows that are fully contained in the blue part of the diagram. 
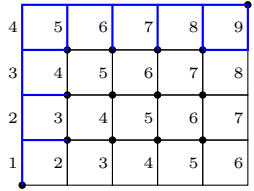
\[\begin{tikzpicture}[scale=0.6]
\draw (0,0) rectangle (1,1);
\draw (0,1) rectangle (1,2);
\draw (0,2) rectangle (1,3);
\draw (0,3) rectangle (1,4);
\draw[] (1,1) rectangle (2,2);
\draw[] (1,0) rectangle (2,1);
\draw[] (1,2) rectangle (2,3);
\draw[] (1,3) rectangle (2,4);
\draw[] (2,1) rectangle (3,2);
\draw[] (2,2) rectangle (3,3);
\draw[] (2,0) rectangle (3,1);
\draw[] (2,3) rectangle (3,4);
\draw[] (3,0) rectangle (4,1);
\draw[] (3,1) rectangle (4,2);
\draw[] (3,2) rectangle (4,3);
\draw[] (3,3) rectangle (4,4);
\draw (4,0) rectangle (5,1);
\draw (4,1) rectangle (5,2);
\draw (4,2) rectangle (5,3);
\draw (4,3) rectangle (5,4);
\draw[fill] (0,0) circle (2pt);
\draw[fill] (1,1) circle (2pt);
\draw[fill] (2,1) circle (2pt);
\draw[fill] (3,1) circle (2pt);
\draw[fill] (4,1) circle (2pt);
\draw[fill] (1,2) circle (2pt);
\draw[fill] (2,2) circle (2pt);
\draw[fill] (3,2) circle (2pt);
\draw[fill] (4,2) circle (2pt);
\draw[fill] (1,3) circle (2pt);
\draw[fill] (2,3) circle (2pt);
\draw[fill] (3,3) circle (2pt);
\draw[fill] (4,3) circle (2pt);
\draw[fill] (5,4) circle (2pt);

\draw[blue,thick] (0,0)--(0,4)--(5,4);
\draw[blue,thick] (0,3)--(1,3);
\draw[blue,thick] (0,2)--(1,2);
\draw[blue,thick] (0,1)--(1,1);
\draw[blue,thick] (1,4)--(1,3);
\draw[blue,thick] (2,4)--(2,3);
\draw[blue,thick] (3,4)--(3,3);
\draw[blue,thick] (4,4)--(4,3);
\draw[blue,thick] (4,3)--(5,3)--(5,4);

\node at (-0.2,0.5) (A) {\tiny 1};
\node at (-0.2,1.5) (A1) { \tiny 2};
\node at (-0.2,2.5) (A2) { \tiny 3};
\node at (-0.2,3.5) (A3) { \tiny 4};
\node at (0.8,0.5) (B) { \tiny 2};
\node at (0.8,1.5) (B1) { \tiny 3};
\node at (0.8,2.5) (B2) { \tiny 4};
\node at (0.8,3.5) (B3) { \tiny 5};
\node at (1.8,0.5) (C) { \tiny 3};
\node at (1.8,1.5) (C1) { \tiny 4};
\node at (1.8,2.5) (C2) { \tiny 5};
\node at (1.8,3.5) (C3) { \tiny 6};
\node at (2.8,0.5) (D) { \tiny 4};
\node at (2.8,1.5) (D1) { \tiny 5};
\node at (2.8,2.5) (D2) { \tiny 6};
\node at (2.8,3.5) (D3) { \tiny 7};
\node at (3.8,0.5) (E) { \tiny 5};
\node at (3.8,1.5) (E1) { \tiny 6};
\node at (3.8,2.5) (E2) { \tiny 7};
\node at (3.8,3.5) (E3) { \tiny 8};
\node at (4.8,0.5) (F) { \tiny 6};
\node at (4.8,1.5) (F1) { \tiny 7};
\node at (4.8,2.5) (F2) { \tiny 8};
\node at (4.8,3.5) (F3) { \tiny 9};
\end{tikzpicture}\]
In general there are $n$ arrows in $C_1$. We label them, starting at the bottom left corner, by $a^1_1,\dots,a^1_n.$ This collection is chosen so that if $(\psi_a) \in  \tilde{\Psi}(\tilde{H}_{n,r})$, then the $\psi_a$,~$a \not \in C_1$ are determined by the $\psi_a$,~$a \in C_1$. This is because the $\psi_a$,~$a \in C_1$, are:
\[\zeta_1,\zeta_1 \zeta_2, \dots, \zeta_1 \cdots \zeta_r, \zeta_{r+1}, \zeta_{r+2}, \dots, \zeta_{n-1}, \zeta_n\]
so that the exponent matrix of the map $\Psi$ composed with the projection $\pi:(\psi_a)_{a \in \LQ_1} \mapsto (\psi_a)_{a \in C_1}$ is an invertible matrix in $\GL(n,\ZZ)$.  This means that for $(\psi_a)_{a} \in \tilde{\Psi}(\tilde{H}_n,r)$, we can express each $\psi_a$,~$a \in \LQ_1$, as a function of the $\psi_a$,~$a \in C_1$. Denote this function by $\xi_a$, which we view as a function on $\tilde{G}$ that depends only on the $\psi_a$ for $a \in C_1$. Then the maps $\xi_a$ provide a way of measuring how far away an element of $\tilde{G}$ is from being in the image of $\tilde{\Psi}(\tilde{H}_{n,r})$:
\[ g=(\psi_a)_{a} \in \tilde{\Psi}(\tilde{H}_{n,r}) \Leftrightarrow \text{$\xi_a(g)=\psi_a$ for all $a \in \LQ_1$.}\]

The second collection of arrows $C_2$ is drawn in green in the example below. It is disjoint from $C_1$. 
\[\begin{tikzpicture}[scale=0.6]
\draw (0,0) rectangle (1,1);
\draw (0,1) rectangle (1,2);
\draw (0,2) rectangle (1,3);
\draw (0,3) rectangle (1,4);
\draw[] (1,1) rectangle (2,2);
\draw[] (1,0) rectangle (2,1);
\draw[] (1,2) rectangle (2,3);
\draw[] (1,3) rectangle (2,4);
\draw[] (2,1) rectangle (3,2);
\draw[] (2,2) rectangle (3,3);
\draw[] (2,0) rectangle (3,1);
\draw[] (2,3) rectangle (3,4);
\draw[] (3,0) rectangle (4,1);
\draw[] (3,1) rectangle (4,2);
\draw[] (3,2) rectangle (4,3);
\draw[] (3,3) rectangle (4,4);
\draw (4,0) rectangle (5,1);
\draw (4,1) rectangle (5,2);
\draw (4,2) rectangle (5,3);
\draw (4,3) rectangle (5,4);
\draw[fill] (0,0) circle (2pt);
\draw[fill] (1,1) circle (2pt);
\draw[fill] (2,1) circle (2pt);
\draw[fill] (3,1) circle (2pt);
\draw[fill] (4,1) circle (2pt);
\draw[fill] (1,2) circle (2pt);
\draw[fill] (2,2) circle (2pt);
\draw[fill] (3,2) circle (2pt);
\draw[fill] (4,2) circle (2pt);
\draw[fill] (1,3) circle (2pt);
\draw[fill] (2,3) circle (2pt);
\draw[fill] (3,3) circle (2pt);
\draw[fill] (4,3) circle (2pt);
\draw[fill] (5,4) circle (2pt);

\draw[green,thick] (0,0)--(1,0)--(1,1);
\draw[green,thick] (2,1)--(1,1);
\draw[green,thick] (2,2)--(1,2);
\draw[green,thick] (2,3)--(1,3);
\draw[green,thick] (3,1)--(2,1);
\draw[green,thick] (3,2)--(2,2);
\draw[green,thick] (3,3)--(2,3);
\draw[green,thick] (4,1)--(3,1);
\draw[green,thick] (4,2)--(3,2);
\draw[green,thick] (4,3)--(3,3);
\draw[green,thick] (1,1)--(1,3);
%\draw[green,thick] (4,3)--(5,3)--(5,4);

\node at (-0.2,0.5) (A) {\tiny 1};
\node at (-0.2,1.5) (A1) { \tiny 2};
\node at (-0.2,2.5) (A2) { \tiny 3};
\node at (-0.2,3.5) (A3) { \tiny 4};
\node at (0.8,0.5) (B) { \tiny 2};
\node at (0.8,1.5) (B1) { \tiny 3};
\node at (0.8,2.5) (B2) { \tiny 4};
\node at (0.8,3.5) (B3) { \tiny 5};
\node at (1.8,0.5) (C) { \tiny 3};
\node at (1.8,1.5) (C1) { \tiny 4};
\node at (1.8,2.5) (C2) { \tiny 5};
\node at (1.8,3.5) (C3) { \tiny 6};
\node at (2.8,0.5) (D) { \tiny 4};
\node at (2.8,1.5) (D1) { \tiny 5};
\node at (2.8,2.5) (D2) { \tiny 6};
\node at (2.8,3.5) (D3) { \tiny 7};
\node at (3.8,0.5) (E) { \tiny 5};
\node at (3.8,1.5) (E1) { \tiny 6};
\node at (3.8,2.5) (E2) { \tiny 7};
\node at (3.8,3.5) (E3) { \tiny 8};
\node at (4.8,0.5) (F) { \tiny 6};
\node at (4.8,1.5) (F1) { \tiny 7};
\node at (4.8,2.5) (F2) { \tiny 8};
\node at (4.8,3.5) (F3) { \tiny 9};
\end{tikzpicture}\]
Taking the target map from $C_2 \to \LQ_0$ gives an injection.

\begin{claim}
The $T \cap \tilde{G}$ orbit of any $g \in \tilde{G}$ contains another group element $h=(\psi_a)_{a \in \LQ_1}$ satisfying, for $a \in C_2$:
\begin{equation}\label{eq:gprime} \psi_a=\xi_a(h).
\end{equation}
In other words, along $a \in C_2$, the relations we need for $h$ to be an element of $\tilde{\Psi}(\tilde{H}_{n,r})$ are satisfied. 
\end{claim}
\begin{proof}[Proof of claim]
There are two types of arrows in $C_2$: horizontal arrows and arrows containing a vertical step, which we call vertical for simplicity.  Each of the arrows in $C_2$ picks out a column in the weight matrix $[b_{va}]$ of $X_{{P_{n,r}}}$. If we order the arrows in $C_1$ from left to right and top to bottom, the resulting sub-matrix of $[b_{va}]$ is the identity matrix along the vertical arrows, and then upper triangular for the rest with $1$s on the diagonal.
  
If $a$ is horizontal, then the equation that $h$ must satisfy is
\[\psi_a=1.\]
Recall that the generators of $T \cap \tilde{G}$ are given by the rows of the weight matrix of the ladder quiver -- there is a row for each vertex, and the entries are the $b_{va}$; see \eqref{eqn:generators}. 
Since the submatrix is upper triangular, we can use the generators  of $T \cap \tilde{G}$ corresponding to the vertices $t(a)$, $a$ horizontal, to move $g$ within its orbit to a  group element satisfying the equations we need for horizontal $a \in C_2$. Call this element $g'=(\psi'_a)$. 

 We now focus on the vertical arrows. Note that none of the generators of $T \cap \tilde{G}$ associated to the vertical arrows act non-trivially on the $\psi_a$,~$a \in C_2$, $a$ horizontal. Label the vertical arrows of $C_2$ from bottom to top as $a^2_1,\dots,a^2_{r-1}$. Then the equations that $g'$ must satisfy are
\[\psi_{a^2_i}=\psi_{a^1_{i+1}}/\psi_{a^1_{i}}.\]
To move $g$ within its orbit so that it satisfies these equations, we start with the last vertical arrow, $a^2_{r-1}$. If we act with the generator associated to $t(a^2_{r-1})$, we can set the ${a^2_{r-1}}$-th coordinate to 
\[\psi'_{a^1_{r}}/\psi'_{a^1_{r-1}}.\]
Note that that generator acts non-trivially on the $a^1_{r}$ and $a^1_{r-1}$ coordinates -- however, it acts on the same weight with each, so we now have  a group element which satisfies the $a^2_{r_1}$ equation \emph{and} the horizontal equations. If we move one step lower, the same thing happens: our group element acts trivially on the ratios $\psi'_{a^1_{r}}/\psi'_{a^1_{r-1}}$ and $\psi'_{a^1_{r-1}}/\psi'_{a^1_{r-2}}$. Continuing downwards, eventually we find $h$ as required. 
\end{proof}

Let's rephrase what we've done so far: for every $[g] \in G$, we have identified a particular lift $g \in \tilde{G}$ satisfying
\begin{equation}\label{eq:gnotprime} \psi_a=\xi_a(g) \text{ for all } a \in C_2.
\end{equation}
To complete the proof, it is enough to show that if $[g] \in G_h$, then the special lift $g$ constructed above is in $\tilde{\Psi}(\tilde{H}_{n,r})$. Since $\tilde{G} \cap T$ acts trivially on the Pl\U cker coordinates, both $g$ and $[g]$ preserve the Pl\U cker relations. In particular, $g$ preserves certain three term Pl\U cker relations: we will define such a relation for every internal vertex in the ladder quiver.

For every internal vertex $v \in \LQ_0$  fix an incomparable pair $\lambda_v$,~$\sigma_v$ which cross at $v$ and agree outside of the $2\times 2$ box centered at $v$. (The argument will not depend on what happens outside the $2 \times 2$ box centered at $v$.)  For example, here are some pairs with $v$ highlighted in red and the box shaded in gray. 
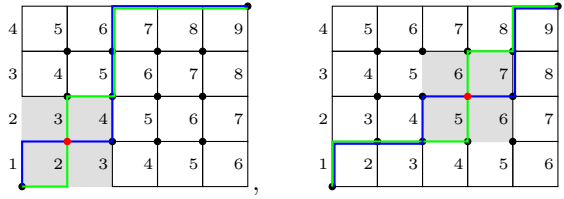
\[\begin{tikzpicture}[scale=0.6]
\draw[fill, lightgray] (0,0) rectangle (1,1);
\draw[fill,lightgray] (0,1) rectangle (1,2);
\draw (0,2) rectangle (1,3);
\draw (0,3) rectangle (1,4);
\draw[fill,lightgray] (1,1) rectangle (2,2);
\draw[fill,lightgray] (1,0) rectangle (2,1);
\draw[] (1,2) rectangle (2,3);
\draw[] (1,3) rectangle (2,4);
\draw[] (2,1) rectangle (3,2);
\draw[] (2,2) rectangle (3,3);
\draw[] (2,0) rectangle (3,1);
\draw[] (2,3) rectangle (3,4);
\draw[] (3,0) rectangle (4,1);
\draw[] (3,1) rectangle (4,2);
\draw[] (3,2) rectangle (4,3);
\draw[] (3,3) rectangle (4,4);
\draw (4,0) rectangle (5,1);
\draw (4,1) rectangle (5,2);
\draw (4,2) rectangle (5,3);
\draw (4,3) rectangle (5,4);
\draw[fill] (0,0) circle (2pt);

\draw[fill] (2,1) circle (2pt);
\draw[fill] (3,1) circle (2pt);
\draw[fill] (4,1) circle (2pt);
\draw[fill] (1,2) circle (2pt);
\draw[fill] (2,2) circle (2pt);
\draw[fill] (3,2) circle (2pt);
\draw[fill] (4,2) circle (2pt);
\draw[fill] (1,3) circle (2pt);
\draw[fill] (2,3) circle (2pt);
\draw[fill] (3,3) circle (2pt);
\draw[fill] (4,3) circle (2pt);
\draw[fill] (5,4) circle (2pt);

\draw[blue,thick] (0,0)--(0,1)--(1,1)--(2,1)--(2,4)--(5,4);
\draw[green,thick] (0,0)--(1,0)--(1,1)--(1,2)--(2.05,2)--(2.05,3.95)--(5,3.95);
\draw[fill,red] (1,1) circle (2pt);

\node at (-0.2,0.5) (A) {\tiny 1};
\node at (-0.2,1.5) (A1) { \tiny 2};
\node at (-0.2,2.5) (A2) { \tiny 3};
\node at (-0.2,3.5) (A3) { \tiny 4};
\node at (0.8,0.5) (B) { \tiny 2};
\node at (0.8,1.5) (B1) { \tiny 3};
\node at (0.8,2.5) (B2) { \tiny 4};
\node at (0.8,3.5) (B3) { \tiny 5};
\node at (1.8,0.5) (C) { \tiny 3};
\node at (1.8,1.5) (C1) { \tiny 4};
\node at (1.8,2.5) (C2) { \tiny 5};
\node at (1.8,3.5) (C3) { \tiny 6};
\node at (2.8,0.5) (D) { \tiny 4};
\node at (2.8,1.5) (D1) { \tiny 5};
\node at (2.8,2.5) (D2) { \tiny 6};
\node at (2.8,3.5) (D3) { \tiny 7};
\node at (3.8,0.5) (E) { \tiny 5};
\node at (3.8,1.5) (E1) { \tiny 6};
\node at (3.8,2.5) (E2) { \tiny 7};
\node at (3.8,3.5) (E3) { \tiny 8};
\node at (4.8,0.5) (F) { \tiny 6};
\node at (4.8,1.5) (F1) { \tiny 7};
\node at (4.8,2.5) (F2) { \tiny 8};
\node at (4.8,3.5) (F3) { \tiny 9};
\end{tikzpicture}, \hspace{6mm}
\begin{tikzpicture}[scale=0.6]
\draw (0,0) rectangle (1,1);
\draw (0,1) rectangle (1,2);
\draw (0,2) rectangle (1,3);
\draw (0,3) rectangle (1,4);
\draw[] (1,1) rectangle (2,2);
\draw[] (1,0) rectangle (2,1);
\draw[] (1,2) rectangle (2,3);
\draw[] (1,3) rectangle (2,4);
\draw[fill,lightgray] (2,1) rectangle (3,2);
\draw[fill,lightgray] (2,2) rectangle (3,3);
\draw[] (2,0) rectangle (3,1);
\draw[] (2,3) rectangle (3,4);
\draw[] (3,0) rectangle (4,1);
\draw[fill,lightgray] (3,1) rectangle (4,2);
\draw[fill, lightgray] (3,2) rectangle (4,3);
\draw[] (3,3) rectangle (4,4);
\draw (4,0) rectangle (5,1);
\draw (4,1) rectangle (5,2);
\draw (4,2) rectangle (5,3);
\draw (4,3) rectangle (5,4);
\draw[fill] (0,0) circle (2pt);

\draw[fill] (2,1) circle (2pt);
\draw[fill] (3,1) circle (2pt);
\draw[fill] (4,1) circle (2pt);
\draw[fill] (1,2) circle (2pt);
\draw[fill] (2,2) circle (2pt);
\draw[fill] (1,1) circle (2pt);
\draw[fill] (4,2) circle (2pt);
\draw[fill] (1,3) circle (2pt);
\draw[fill] (2,3) circle (2pt);
\draw[fill] (3,3) circle (2pt);
\draw[fill] (4,3) circle (2pt);
\draw[fill] (5,4) circle (2pt);

\draw[green,thick] (0,0)--(0,1)--(2,1)--(3,1)--(3,2)--(3,3)--(4,3)--(4,4)--(5,4);
\draw[blue,thick] (0.05,0)--(0.05,0.95)--(2,0.95)--(2,2)--(3,2)--(4.05,2)--(4.05,3.95)--(5,3.95);
\draw[fill,red] (3,2) circle (2pt);

\node at (-0.2,0.5) (A) {\tiny 1};
\node at (-0.2,1.5) (A1) { \tiny 2};
\node at (-0.2,2.5) (A2) { \tiny 3};
\node at (-0.2,3.5) (A3) { \tiny 4};
\node at (0.8,0.5) (B) { \tiny 2};
\node at (0.8,1.5) (B1) { \tiny 3};
\node at (0.8,2.5) (B2) { \tiny 4};
\node at (0.8,3.5) (B3) { \tiny 5};
\node at (1.8,0.5) (C) { \tiny 3};
\node at (1.8,1.5) (C1) { \tiny 4};
\node at (1.8,2.5) (C2) { \tiny 5};
\node at (1.8,3.5) (C3) { \tiny 6};
\node at (2.8,0.5) (D) { \tiny 4};
\node at (2.8,1.5) (D1) { \tiny 5};
\node at (2.8,2.5) (D2) { \tiny 6};
\node at (2.8,3.5) (D3) { \tiny 7};
\node at (3.8,0.5) (E) { \tiny 5};
\node at (3.8,1.5) (E1) { \tiny 6};
\node at (3.8,2.5) (E2) { \tiny 7};
\node at (3.8,3.5) (E3) { \tiny 8};
\node at (4.8,0.5) (F) { \tiny 6};
\node at (4.8,1.5) (F1) { \tiny 7};
\node at (4.8,2.5) (F2) { \tiny 8};
\node at (4.8,3.5) (F3) { \tiny 9};
\end{tikzpicture}\]
Fix an interior vertex $v$. Locally about vertex $v$, $\lambda_v$ and $\sigma_v$ look like:
\begin{equation}\label{eq:sigma} 
  \begin{aligned}
    \begin{tikzpicture}[scale=1.0]
        \draw (0,0) rectangle (1,1);
        \draw (0,1) rectangle (1,2);

        \draw[] (1,1) rectangle (2,2);
        \draw[] (1,0) rectangle (2,1);
        \draw[fill] (1,1) circle (2pt);
        \draw[blue,thick] (0,0)--(0,1)--(1,1)--(2,1)--(2,2);
        \draw[green,thick] (0,0)--(1,0)--(1,1)--(1,2)--(2,2);
        \draw[fill,red] (1,1) circle (2pt);

        \node at (-0.3,0.5) (A) {\tiny $i$};
        \node at (-0.3,1.5) (A1) { \tiny $i+1$};
        \node at (0.7,0.5) (B) { \tiny $i+1$};
        \node at (0.7,1.5) (B1) { \tiny $i+2$};
        \node at (1.7,0.5) (C) { \tiny $i+2$};
        \node at (1.7,1.5) (C1) { \tiny $i+3$};

        \node at (1.2,0.8) (A) {\small v};
      \end{tikzpicture}
    \end{aligned}
\end{equation}
By definition, $\sigma_v \wedge \lambda_v$ and $\sigma_v \vee \lambda_v$ agree with $\lambda_v, \sigma_v$ outside of the box centered at $v$, and locally at $v$ they look like:
\[\begin{tikzpicture}[scale=1.0]
\draw (0,0) rectangle (1,1);
\draw (0,1) rectangle (1,2);

\draw[] (1,1) rectangle (2,2);
\draw[] (1,0) rectangle (2,1);
\draw[fill] (1,1) circle (2pt);
\draw[blue,thick] (0,0)--(0,1)--(1,1)--(1,2)--(2,2);
\draw[green,thick] (0,0)--(1,0)--(1,1)--(2,1)--(2,2);
\draw[fill,red] (1,1) circle (2pt);

\node at (-0.3,0.5) (A) {\tiny $i$};
\node at (-0.3,1.5) (A1) { \tiny $i+1$};
\node at (0.7,0.5) (B) { \tiny $i+1$};
\node at (0.7,1.5) (B1) { \tiny $i+2$};
\node at (1.7,0.5) (C) { \tiny $i+2$};
\node at (1.7,1.5) (C1) { \tiny $i+3$};

\node at (1.2,0.8) (A) {\small v};
\end{tikzpicture}\]

The incomparable pair $\sigma_v, \lambda_v$ give rise to a 3-term Pl\U cker relation: 
\[p_{\alpha_v} p_{\beta_v}-p_{\lambda_v} p_{\sigma_v}+p_{\sigma_v \wedge \lambda_v} p_{\sigma_v \vee \lambda_v}=0.\]
The partitions $\alpha_v$,~$\beta_v$ are characterised by agreeing with $\sigma$,~$\lambda$ outside of the box centered at $v$, and locally at $v$ they look like: 
\begin{equation} \label{eq:sigma2}
\begin{aligned}
\begin{tikzpicture}[scale=1.0]
\draw (0,0) rectangle (1,1);
\draw (0,1) rectangle (1,2);

\draw[] (1,1) rectangle (2,2);
\draw[] (1,0) rectangle (2,1);
\draw[fill] (1,1) circle (2pt);
\draw[blue,thick] (0,0)--(0,2)--(2,2);
\draw[green,thick] (0,0)--(2,0)--(2,2);
\draw[fill,red] (1,1) circle (2pt);

\node at (-0.3,0.5) (A) {\tiny $i$};
\node at (-0.3,1.5) (A1) { \tiny $i+1$};
\node at (0.7,0.5) (B) { \tiny $i+1$};
\node at (0.7,1.5) (B1) { \tiny $i+2$};
\node at (1.7,0.5) (C) { \tiny $i+2$};
\node at (1.7,1.5) (C1) { \tiny $i+3$};

\node at (1.2,0.8) (A) {\small v};
\end{tikzpicture}
\end{aligned}
\end{equation}
One can easily read off from this picture that this relation is homogeneous with respect to the $\Psi(H_{n,r})$-action. Let $S_{v}$ be the collection of arrows  of $\LQ$ whose intersection with the $2\times 2$ box in \eqref{eq:sigma} is contained in the blue or green highlighted paths in \eqref{eq:sigma}. Let $T_v$ be the collection of arrows  of $\LQ$ whose intersection with the $2 \times 2$ box in \eqref{eq:sigma2} is contained in the blue or green highlighted paths in \eqref{eq:sigma2}. If $g=(\psi_a) \in \tilde{G}$ preserves this relationship, then
\begin{equation}\label{eq:vhomog} \prod_{a \in S_{v}} \psi_a= \prod_{a \in T_{v}} \psi_a.\end{equation}
Now fix  some $[g] \in G_h$ with lift $g \in \tilde{G} $ constructed as in \eqref{eq:gnotprime}. The Pl\U cker relation constructed for any $v$ is homogeneous for $g=(\psi_a)$, so \eqref{eq:vhomog} holds for all $v$. 

Consider the following diagram, which highlights both $C_1$ and $C_2$. 
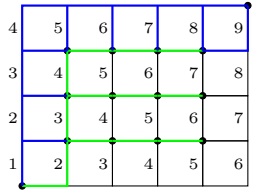
\[\begin{tikzpicture}[scale=0.6]
\draw (0,0) rectangle (1,1);
\draw (0,1) rectangle (1,2);
\draw (0,2) rectangle (1,3);
\draw (0,3) rectangle (1,4);
\draw[] (1,1) rectangle (2,2);
\draw[] (1,0) rectangle (2,1);
\draw[] (1,2) rectangle (2,3);
\draw[] (1,3) rectangle (2,4);
\draw[] (2,1) rectangle (3,2);
\draw[] (2,2) rectangle (3,3);
\draw[] (2,0) rectangle (3,1);
\draw[] (2,3) rectangle (3,4);
\draw[] (3,0) rectangle (4,1);
\draw[] (3,1) rectangle (4,2);
\draw[] (3,2) rectangle (4,3);
\draw[] (3,3) rectangle (4,4);
\draw (4,0) rectangle (5,1);
\draw (4,1) rectangle (5,2);
\draw (4,2) rectangle (5,3);
\draw (4,3) rectangle (5,4);
\draw[fill] (0,0) circle (2pt);
\draw[fill] (1,1) circle (2pt);
\draw[fill] (2,1) circle (2pt);
\draw[fill] (3,1) circle (2pt);
\draw[fill] (4,1) circle (2pt);
\draw[fill] (1,2) circle (2pt);
\draw[fill] (2,2) circle (2pt);
\draw[fill] (3,2) circle (2pt);
\draw[fill] (4,2) circle (2pt);
\draw[fill] (1,3) circle (2pt);
\draw[fill] (2,3) circle (2pt);
\draw[fill] (3,3) circle (2pt);
\draw[fill] (4,3) circle (2pt);
\draw[fill] (5,4) circle (2pt);

\draw[green,thick] (0,0)--(1,0)--(1,1);
\draw[green,thick] (2,1)--(1,1);
\draw[green,thick] (2,2)--(1,2);
\draw[green,thick] (2,3)--(1,3);
\draw[green,thick] (3,1)--(2,1);
\draw[green,thick] (3,2)--(2,2);
\draw[green,thick] (3,3)--(2,3);
\draw[green,thick] (4,1)--(3,1);
\draw[green,thick] (4,2)--(3,2);
\draw[green,thick] (4,3)--(3,3);
\draw[green,thick] (1,1)--(1,3);

\draw[blue,thick] (0,0)--(0,4)--(5,4);
\draw[blue,thick] (0,3)--(1,3);
\draw[blue,thick] (0,2)--(1,2);
\draw[blue,thick] (0,1)--(1,1);
\draw[blue,thick] (1,4)--(1,3);
\draw[blue,thick] (2,4)--(2,3);
\draw[blue,thick] (3,4)--(3,3);
\draw[blue,thick] (4,4)--(4,3);
\draw[blue,thick] (4,3)--(5,3)--(5,4);

\node at (-0.2,0.5) (A) {\tiny 1};
\node at (-0.2,1.5) (A1) { \tiny 2};
\node at (-0.2,2.5) (A2) { \tiny 3};
\node at (-0.2,3.5) (A3) { \tiny 4};
\node at (0.8,0.5) (B) { \tiny 2};
\node at (0.8,1.5) (B1) { \tiny 3};
\node at (0.8,2.5) (B2) { \tiny 4};
\node at (0.8,3.5) (B3) { \tiny 5};
\node at (1.8,0.5) (C) { \tiny 3};
\node at (1.8,1.5) (C1) { \tiny 4};
\node at (1.8,2.5) (C2) { \tiny 5};
\node at (1.8,3.5) (C3) { \tiny 6};
\node at (2.8,0.5) (D) { \tiny 4};
\node at (2.8,1.5) (D1) { \tiny 5};
\node at (2.8,2.5) (D2) { \tiny 6};
\node at (2.8,3.5) (D3) { \tiny 7};
\node at (3.8,0.5) (E) { \tiny 5};
\node at (3.8,1.5) (E1) { \tiny 6};
\node at (3.8,2.5) (E2) { \tiny 7};
\node at (3.8,3.5) (E3) { \tiny 8};
\node at (4.8,0.5) (F) { \tiny 6};
\node at (4.8,1.5) (F1) { \tiny 7};
\node at (4.8,2.5) (F2) { \tiny 8};
\node at (4.8,3.5) (F3) { \tiny 9};
\end{tikzpicture}\]
Consider the top left interior vertex $v_0$. Note that there is precisely one arrow $a_0$ in $S_{v_0} \cup T_{v_0}$ that is not contained in $C:=C_1 \cup C_2$. Using \eqref{eq:vhomog} for $v_0$ one can see that $\psi'_{a_0}$ is determined by the $\psi_a$,~$a \in C$ -- and so in fact, is determined by the $\psi_a$,~$a \in C_1$. Abusing notation, replace $C$ by $C \cup \{a_0\}$. Now consider the vertex $v_1$ directly below $v_0$, and note that there is precisely one arrow $a_1$ in $S_{v_1} \cup T_{v_1}$ that is not contained in (the new) $C$. Applying \eqref{eq:vhomog} for $v_1$ we see that $\psi'_{a_1}$ is determined by the $\psi_a$,~$a \in C$ and hence by $\psi_a'$,~$a \in C_1$. Applying this argument repeatedly, moving down columns from left to right, we conclude that $g$ is uniquely characterized by: 
\begin{itemize}
\item the values $(\psi_a)_{a \in C_1}$
\item equation \eqref{eq:gnotprime}
\item preserving the Pl\U cker relations. 
\end{itemize}
We now use the $\psi_a$,~$a \in C_1$, to define an element of $\tilde{\Psi}(\tilde{H}_{n,r})$, and deduce that this element agrees with $g$. Writing the relation $\prod_{a \in \LQ_1} \psi_a=1$ in terms of $\psi_a$,~$a \in C_1$, we see that
\[\prod_{i=1}^r \left(\frac{\psi_{a^2_{i}}}{\psi_{a^2_{i-1}}}\right)^r \prod_{i=r+1}^n \psi_{a^2_{i}}^r=1.\]
So   $(\xi_a((\psi_a)_{a \in C_1}))$ defines an element of $\tilde{\Psi}(\tilde{H}_{n,r}) \subset \tilde{G}$ which agrees with $g$ along $a \in C_1$, and satisfies the other two conditions. It follows that $g=\xi((\psi_a)_{a \in C_1})$, and in particular $g \in \tilde{\Psi}(\tilde{H}_{n,r})$. Finally we conclude that $[g] \in \Psi(H_{n,r})$ as claimed. 
\end{proof}

%% TODO got to here

To translate this combinatorial statement to geometry, we extend the $G$-action from $X_P$ to the family which gives the toric degeneration. The degeneration is given by the Pl\U cker relations. To illustrate, we describe the $\Gr(5,2)$ example:
\begin{eg}\label{eg:gr52deg} When $n=5$,~$r=2$, the family is given by the 5 polynomials
\[c p_{12} p_{34}-p_{13}p_{24}+p_{14}p_{23},\]
\[c p_{12} p_{35}-p_{13}p_{25}+p_{15}p_{23},\]
\[c p_{12} p_{45}-p_{14}p_{25}+p_{15}p_{24},\]
\[c p_{13} p_{45}-p_{14}p_{35}+p_{15}p_{34},\]
\[c p_{23} p_{45}-p_{24}p_{35}+p_{25}p_{34}.\]
that is, this is a family in $\PP^{9} \times \CC$, where the coordinate on the second factor is $c$. The fiber over $c=0$ is the degenerate toric variety $X_{P_{5,2}}$, and the generic fiber is isomorphic to the Grassmannian $\Gr(5,2)$. To make this family $G$-equivariant, we need to extend the action of $G$ to the coefficient~$c$. Label the vertices of the ladder diagram as shown:
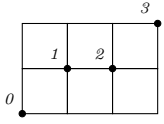
\[
  \begin{tikzpicture}[scale=0.6]
    \draw (0,0) rectangle (1,1);
    \draw (0,1) rectangle (1,2);
    \draw (1,1) rectangle (2,2);
    
    \draw (1,0) rectangle (2,1);
    \draw (2,1) rectangle (3,2);
    
    \draw (2,0) rectangle (3,1);
    \draw[fill] (0,0) circle (2pt);
    \draw[fill] (1,1) circle (2pt);
    \draw[fill] (2,1) circle (2pt);
    \draw[fill] (3,2) circle (2pt);

    \node[label={\tiny 0}] at (-0.3,-0.2) {};
    \node[label={\tiny 1}] at (1-0.3,1-0.2) {};
    \node[label={\tiny 2}] at (2-0.3,1-0.2) {};
    \node[label={\tiny 3}] at (3-0.3,2-0.2) {};
  \end{tikzpicture}
\]
There are nine arrows in the ladder quiver, which we label
\[a_{0,1},a_{0,1}',a_{0,2}, a_{0,3},a_{0,3}',a_{1,2},a_{1,3},a_{2,3},a_{2,3}',\]
where $a_{i,j}$ is an arrow from $i \to j$, and we add a $'$ to the lower arrow if there are two such arrows. 
Consider the first equation above. The weight of the first monomial is $\psi_{a_{0,3}} \psi_{a_{0,2}} \psi_{a_{2,3}}$; the weight of both the second and third monomial is $\psi_{a_{0,1}} \psi_{a_{1,3}} \psi_{a_{0,1}'} \psi_{a_{1,2}} \psi_{a_{2,3}}$, so for $c$ to be equivariant in this equation, it should have weight
\[ (\psi_{a_{0,3}} \psi_{a_{0,2}} \psi_{a_{2,3}})^{-1} \psi_{a_{0,1}} \psi_{a_{1,3}} \psi_{a_{0,1}'} \psi_{a_{1,2}} \psi_{a_{2,3}}.\]
However, as we go through the remaining five relations, the required weight for $c$ changes, so we are forced to introduce multiple coefficients. Consider the polynomials
\[c_1 p_{12} p_{34}-p_{13}p_{24}+p_{14}p_{23},\]
\[c_2 p_{12} p_{35}-p_{13}p_{25}+p_{15}p_{23},\]
\[c_3 p_{12} p_{45}-p_{14}p_{25}+p_{15}p_{24},\]
\[c_4 p_{13} p_{45}-p_{14}p_{35}+p_{15}p_{34},\]
\[c_5 p_{23} p_{45}-p_{24}p_{35}+p_{25}p_{34}.\]
We view the toric degeneration as a family  in $Z\subset \PP^{9} \times \CC^5$, where the second factor has coordinates $(c_i)$, and the equations cutting out $Z$ are above.  Computing the required weights of $c_1,\dots,c_5$, we note that $c_1$ has the same weight as $c_2$ and $c_4$ has the same weight as $c_5$. The variable $c_3$ has the weight of the monomial $c_1 c_4$. To get rid of the excess variables, we add the equations
\[ c_1=c_2, \: c_4=c_5, \: c_3= c_1 c_4.\]

Consider the quotient $Z/G$. The special fiber -- when $c_i=0$ -- is $X_{P_{n,r}}/G$. Now consider a generic fiber of $Z$. The equation $c_i=1$ is not homogeneous for the $G$-action; however, we can consider the disconnected subvariety $Z_{1}$ cut out of $Z$ by the equations $c_i^5=1$, $i \in \{1,\ldots,5\}$.
The subgroup of $G$ that preserves a given component of $Z_{1}$ is precisely $H_{5,2}$. The number of components of $Z_1$ is 25. The group $G/H_{5,2}$ has order $25$.  Comparing group sizes, it follows that $Z_{1}/G=\Gr(5,2)/H_{5,2}$. 
\end{eg}
The situation is slightly more complicated in the general case.  As in the example, $X_P$ is the special fiber of a family of varieties 
\begin{equation}\label{eq:tildefamily}
\tilde{Z} \subset\PP^{\binom{n}{r}-1} \times \CC^m \xrightarrow{\pi} \CC^m.
\end{equation} 
 We denote points in $\CC^m$ as $\underline{c}$.  The coordinates of $\underline{c}$ give coefficients for \emph{all} monomials in the Pl\U cker relations that disappear under the degeneration, and $\tilde{Z}_{\underline{c}}$ is $\tilde{Z} \cap  \pi^{-1}(\underline{c})$. If $\underline{c}_1:=(1,\dots,1)$ and $\underline{c}_0:=(0,\dots,0)$, then $\tilde{Z}_{\underline{c}_1}=\Gr(n,r)$ and  $\tilde{Z}_{\underline{c}_0}=X_{P_{n,r}}$. The generic fiber is isomorphic to $\Gr(n,r)$. 

We can extend the action of $G$ to $\PP^{\binom{n}{r}-1} \times \CC^m \to \CC^m$ by defining the action on $\CC^m$ to be the one that ensures $\tilde{Z}$ is preserved by $G$.  Note that  $G_h$ is the subgroup of $G$ that acts trivially on $\CC^m$. As in Example \ref{eg:gr52deg}, in general some of these variables are redundant. In fact, the proof of Theorem \ref{thm:groupiso} shows that we need only the special $2 \times 2$ three-term Pl\U cker relations associated to each internal vertex in the ladder quiver. More precisely, the proof shows that if $g \in G$ preserves these relations, then it in fact preserves all Pl\U cker relations. Let $d_v$, where $v$ is an internal vertex, denote the these coefficients. Then for any other coefficient $c$, the weight of $c$ under the $G$-action can be expressed as a ratio of the $d_v$. Whatever these ratios are, they give a collection of polynomial equations between the $c_i$, which we call
\begin{equation}\label{eq:poly1}
\mathcal{P}_1.
\end{equation}

Before we can define the family we want, we need to cut down by one last equation, which is non-trivial when $\gcd(n,r)>1$. (It did not occur in our earlier example because $5$ and $2$ are coprime.) This will be a $G$-equivariant equation, of course, and will take some time to describe. 
\begin{mydef} Fix integers $n>r>0$, and set $k=n-r$ and $d=\gcd(n,r)$.  Then there exist $s$,~$t$ such that $r=s d$ and $k=t d$. The ladder diagram $\LQ_{n,r}$ can be subdivided into an $s \times t$ grid of $d \times d$ squares. The set of $d$-diagonal vertices in $\LQ_0$ is the set of $v \in \LQ_0$ such that $v$ lies on the diagonal of one of the $d \times d$ squares. To a $d$-diagonal vertex $v$ lying in a $d \times d$ square $B$, we associate an incomparable pair of partitions $\sigma_{B,v}$,~$\mu_{B,v}$. As before, we only specify these partitions locally,  insisting that they agree outside of $B$. The partitions are specified by the statement that their union contains both the vertical and horizontal line (contained in $B$) passing through the vertex $v$,  and that both partitions pass through the bottom left and top right vertices of $B$. We draw the two partitions below. 
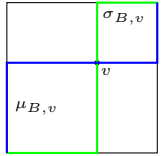
\[\begin{tikzpicture}[scale=0.4]
\draw (0,0) rectangle (5,5);

\draw[fill] (3,3) circle (2pt);

\draw[green,thick] (0,0)--(3,0)--(3,5)--(5,5);
\draw[blue,thick] (0,0)--(0,3)--(5,3)--(5,5);
\node at (3.3,2.7) (A) {\tiny $v$};
\node at (3.9,4.5) (B) {\tiny $\sigma_{B,v}$};
\node at (1,1.5) (C) {\tiny $\mu_{B,v}$};
\end{tikzpicture}\]
We denote by $\sigma_B$ and $\mu_B$ the two partitions that agree with $\sigma_{B,v}$ and $\mu_{B,v}$ away from $B$, and follow the border of $B$:
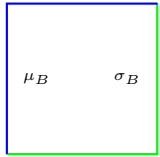
\[\begin{tikzpicture}[scale=0.4]
\draw (0,0) rectangle (5,5);
\draw[green,thick] (0,0)--(5,0)--(5,5);
\draw[blue,thick] (0,0)--(0,5)--(5,5);

\node at (4,2.5) (B) {\tiny $\sigma_{B}$};
\node at (1,2.5) (C) {\tiny $\mu_{B}$};
\end{tikzpicture}\]
\end{mydef}
\begin{lem}\label{lem:othercoeff} Let $v$ be a $d$-diagonal vertex in the ladder diagram $\LQ_{n,r}$. Then there is a Pl\U cker relation that contains both $p_{\sigma_B} p_{\mu_B}$ and $p_{\sigma_{B,v}} p_{\mu_{B,v}}$. 
\end{lem}
\begin{proof}
  Suppose that $v$ is the $a$th  $d$-diagonal vertex in the box $B$, and suppose that the label on the bottom left vertical step of $B$ is $i$. Then $I(\sigma_{B,v})$ differs from $I(\mu_{B,v})$ only in the range $M:=\{ j \mid i \leq j \leq i+2d-1\}$.   It is easy to see that 
\[I(\sigma_{B,v}) \cap M = \{i,i+1,\dots,i+a-1, i+a+d,\dots, i+2d-1\},\]
and
\[I(\mu_{B,v}) \cap M = \{i+a,\dots,i+2a-1,i+2a,\dots,i+a+d-1\}.\]
If $I(\sigma_{B,v})=\{s_1 < \cdots <s_r\}$ and $I(\mu_{B,v})=\{t_1<\cdots< t_r\}$, then choose $p$ such that $s_p=i+a-1$. So we can obtain $I(\sigma_B)$ and $I(\mu_B)$ from $I(\sigma_{B,v})$ and $I(\mu_{B,v})$ by swapping $s_{p+1}<\cdots<s_{d-a+p}$ with $t_{p+1}<\dots<t_{d-a+p}$.  This implies that there is a Pl\U cker relation where both terms appear. For completeness, we construct a relation using \cite[equation 3.1.3]{manivel} in the case when $2a \geq d$ and in the case where $2a<d$.

In each case, we will identify subsets of indices $J_s=\{s_{i_1},\dots,s_{i_l}\}$ and $J_{t}=\{t_{j_1},\dots,t_{j_{r-l}}\}$. Let $S$ be the symmetric group permuting the symbols in the union $J_s \cup J_t$, and $S_1 \times S_2$ the subgroup of $S$ that preserves $I(\sigma_{B,v})$ and $I(\mu_{B,v})$. For $\omega \in S/(S_1 \times S_2)$, set
\[\omega(I(\sigma_{B,v}))=\{\omega(s_{i}): s_i \in J_s \} \cup \{s_i: s_i \not \in J_s\},\]
and 
\[\omega(I(\mu_{B,v}))=\{\mu(t_{i}): t_i \in J_t \} \cup \{t_i: t_i \not \in J_t\}.\]
Then for any choice of $J_s$ and $J_t$, the following equation is a Pl\U cker relation:
\[\sum_{\omega \in S/S_1 \times S_2}\sign(\omega) p_{\omega(I(\sigma_{B,v}))} p_{\omega(I(\mu_{B,v}))}.\]
When $2a \geq d$, we choose 
\[J_s=\{i+a+d,\dots,i+2d-1\},\: J_t=\{i+k,\dots,i+2a-1\} \cup M.\]
When $2a < d$, we choose 
\[J_s=\{i,\dots,i+a-1\}\cup M,\: J_t=\{i+2k,\dots,i+d+a-1\}.\]
Note that for each choice the terms $p_{\sigma_{B,v}} p_{\mu_{B,v}}$ and $p_{\sigma_{B}} p_{\mu_{B}}$ appear with non-zero coefficient as claimed.
\end{proof}
From the proof of Theorem \ref{thm:groupiso}, the weight of the ratio $p_{\sigma_{B,v}} p_{\mu_{B,v}}/(p_{\sigma_{B}} p_{\mu_{B}})$ can be expressed as a ratio of the weights of $d_w$ where $w$ ranges over the internal vertices. Let $f_{B,v}(d_w)$ be this ratio.
\begin{lem} The group $G$ acts trivially on the product  $(\prod_{(B,v)} f_{B,v})^{\frac{n}{d}}$, where $(B,v)$ ranges over internal diagonal vertices in the $d \times d$ boxes $B$ of the ladder diagram. 
\end{lem}
\begin{proof} When $d=1$, the statement is trivial. To prove the lemma, we will show that the weight of $\prod_{(B,v)} f_{B,v}$ is an order $n/d$ element of $G$. We will describe this element precisely. Let $r=sd$ and $n-r=td$.  We have divided the ladder diagram into a $s \times t$ grid of $d\times d$ boxes. Consider the collection $C$ of arrows entirely contained in the border of the $d \times d$ boxes. Then we will show that the weight of $\prod_{(B,v)} f_{B,v}$
is
\begin{equation} \label{eq:weight} \prod_{a \in C} \zeta_a^{-d}.\end{equation}
This is an order $n/d$ element of $G$ as desired. The easiest way to see that this weight is correct is to visualize the weights using paths in the ladder diagram. We will indicate a weight by drawing paths in the ladder diagram, possibly in reverse, and allowing multiple arrows. Each arrow appearing contributes a weight of $\zeta_a^{b^+-b^-}$, where $b^+$ is the number of times $a$ appears in the correct direction and $b^-$ is the number of times it appears in the reverse direction. For example, the weight of $f_{B,v}$ is pictured as
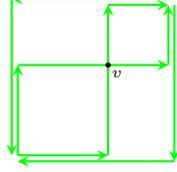
\[\begin{tikzpicture}[scale=0.4]
\draw[green,thick,-stealth] (0,0)--(3,0);
\draw[green,thick,-stealth]  (3,0)--(3,5);
\draw[green,thick,-stealth]  (3,5)--(5,5);
\draw[green,thick,-stealth] (0,0)--(0,3);
\draw[green,thick,-stealth] (0,3)--(5,3);
\draw[green,thick,-stealth](5,3)--(5,5);
\draw[green,thick,stealth-] (0,-0.2)--(5.2,-0.2);
\draw[green,thick,stealth-] (5.2,-0.2)--(5.2,5);
\draw[green,thick,stealth-] (-0.2,0)--(-0.2,5.2);
\draw[green,thick,stealth-] (-0.2,5.2)--(5,5.2);
\draw[fill] (3,3) circle (2pt);
\node at (3.3,2.7) (A) {\tiny $v$};
\end{tikzpicture}\]
Now consider the weight of the product of $f_{B,v}$ where $B$ is fixed: $\prod_{v} f_{B,v}$: pictorially, this corresponds to overlaying the diagrams for each of internal vertices $v$. Note that the arrows going in positive direction have no intersection except on the border, and every such arrow is covered exactly once. Depending on where $B$ is located in the ladder diagram, there could be vertices along the borders of $B$ or not, which changes how many times each arrow appears. If there are no vertices on a particular edge, then there are $d-1$ arrows in the reverse direction along that edge (one for each internal vertex). If there are border vertices, then the number of arrows follows a pattern as below (where $-k$ indicates $k$ arrows in the reverse direction of the arrow directly below the label):
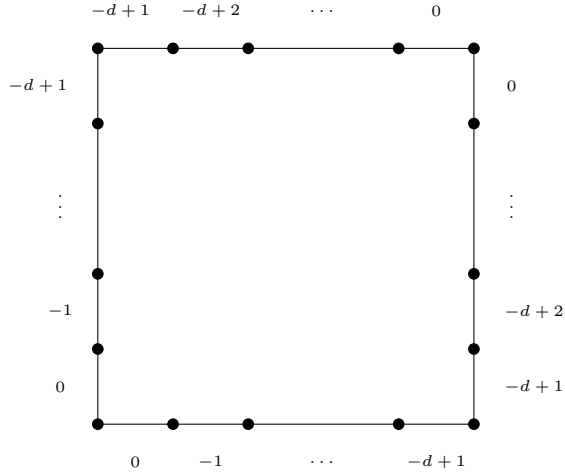
\[\begin{tikzpicture}[scale=1]
\draw[fill] (0,0) circle (2pt);
\draw[fill] (0,1) circle (2pt);
\draw[fill] (0,2) circle (2pt);
\draw[fill] (0,4) circle (2pt);
\draw[fill] (0,5) circle (2pt);

\draw[fill] (5,0) circle (2pt);
\draw[fill] (5,1) circle (2pt);
\draw[fill] (5,2) circle (2pt);
\draw[fill] (5,4) circle (2pt);
\draw[fill] (5,5) circle (2pt);

\draw[fill] (1,0) circle (2pt);
\draw[fill] (2,0) circle (2pt);
\draw[fill] (4,0) circle (2pt);

\draw[fill] (1,5) circle (2pt);
\draw[fill] (2,5) circle (2pt);
\draw[fill] (4,5) circle (2pt);

\draw (0,0) rectangle (5,5);

\node at (-0.5,0.5) (A) {\tiny$0$};
\node at (-0.5,1.5) (A1) {\tiny$-1$};
\node at (-0.5,3) (A2) {\tiny$\vdots$};
\node at (-0.8,4.5) (A3) {\tiny$-d+1$};

\node at (5.5,4.5) (B) {\tiny$0$};
\node at (5.5,3) (B1) {\tiny$\vdots$};
\node at (5.8,1.5) (B2) {\tiny $-d+2$};
\node at (5.8,0.5) (B3) {\tiny$-d+1$};

\node at (0.5,-0.5) (C) {\tiny$0$};
\node at (1.5,-0.5) (C1) {\tiny$-1$};
\node at (3,-0.5) (C2) {\tiny$\cdots$};
\node at (4.5,-0.5) (C3) {\tiny$-d+1$};

\node at (0.3,5.5) (D) {\tiny$-d+1$};
\node at (1.5,5.5) (D1) {\tiny$-d+2$};
\node at (3,5.5) (D2) {\tiny$\cdots$};
\node at (4.5,5.5) (D3) {\tiny$0$};
\end{tikzpicture}\]
The weight of the product over all boxes $B$ -- i.e. the weight of $\prod_{(B,v)} f_{B,v}$ -- is read off by combining all of the $d \times d$ boxes to cover the ladder diagram. It is immediate from the above diagram that this gives us the weight
\[\prod_{a \in C} \zeta_a^{-(d-1)} \prod_{a \not \in C} \zeta_a.\]
Since $\prod_{a} \zeta_a=1$, this shows that the weight is \eqref{eq:weight} as desired. 
\end{proof}
The lemma states that the action of $G$ on $(\prod_{(B,v)} f_{B,v})^{\frac{n}{d}}$ is trivial, so the equation $(\prod_{(B,v)} f_{B,v})^{\frac{n}{d}}=1$ is homogeneous. If we wish, since the left hand side is a monomial, we can clear the denominator to get a polynomial equation. We call the set containing this single polynomial equation
\begin{equation}\label{eq:poly2}
\mathcal{P}_2.
\end{equation}
This is the last equation we need to define our family. 

\begin{mydef} Let $Z \subset \PP^{\binom{n}{r}-1} \times B \xrightarrow{\pi} B$ denote the family cut out by the following three types of the equations:
\begin{enumerate}
\item the Pl\U cker relations,
\item the polynomial equations $\mathcal{P}_1$ from \eqref{eq:poly1},
\item the polynomial equation $\mathcal{P}_2$ from \eqref{eq:poly2}.
 \end{enumerate}
 The last two equations cut out $B \subset \CC^m$. 
\end{mydef}
By construction, $Z$ is a $G$-invariant algebraic variety. We use notation as for the family $\tilde{Z}$ in \eqref{eq:tildefamily}.  Note that the fiber $Z_{\underline{c}_0}$ is invariant under the $G$ action, but $Z_{\underline{c}_1}$ is not. However, since the action of $G$ on on $\CC^m$ is of order $n$, the disconnected variety $\hat{Z}_{\underline{b}}:=Z \cap \pi^{-1}(\{(c_1,\dots,c_m) \in B: c_i^n=b_i\})$ is $G$-invariant for any $\underline{b} \in \CC^m$. We will consider the quotient family $Z/G \to B/G$.

The next result gives the precise sense in which $\Gr(n,r)/H_{n,r}$ is a smoothing of $X_{P_{n,r}}/G$. 
\begin{thm}\label{thm:Gsmooth} Suppose $b$ is generic (i.e. $b_i \neq 0$ for all $i$). Then $\hat{Z}_{\underline{b}}/G$ is a fiber of the quotient family $Z/G \to B/G$, and $\hat{Z}_{\underline{b}}/G \cong \Gr(n,r)/H_{n,r}$. 
\end{thm}
\begin{proof} For simplicity, we will set $b_i=1$ for all $i$, and denote $\hat{Z}_{\underline{b}}/G$ by $\hat{Z}_1/G$. The variety $\hat{Z}_1$ is disconnected, with components we can label $A_1 \sqcup \cdots \sqcup A_s$. One component, which we label $A_1$, is the Grassmannian $\Gr(n,r)$. The quotient can be studied in two steps:
\[\hat{Z}_1/G=(\bigsqcup A_i/G_h)/(G/G_h).\]
Note that the stabilizer of any element of $\bigsqcup A_i/G_h$ under the $G/G_h$-action is trivial, as only $G_h$ preserves the components. Thus the size of an orbit of a point in $\bigsqcup A_i/G_h$ is
\[|G/G_h|=n^{r(n-r)-1}/(n^{n-2} d)=n^{r(n-r)-n} n/d=n^{(r-1)(n-r-1)-1} n/d.\]
In particular, the number of components $s$ is at least $n^{(r-1)(n-r-1)-1} n/d$. On first glance, it may appear that $\hat{Z}_1$ has $n^m$ components, but in fact, the equations from $\mathcal{P}_1$ and $\mathcal{P}_2$ cut this down considerably. The equations from $\mathcal{P}_1$ mean that we can regard all coefficients $c_i$ as a function of the coefficients $d_v$, where $v$ is an internal vertex of the ladder quiver, so there are at most $n^{(r-1)(n-r-1)}$ components. The single equation in $\mathcal{P}_2$ is of the form $f^{n/d},$ which means that in fact the largest possible number of components is $n^{(r-1)(n-r-1)-1} n/d$. So 
\[s=n^{(r-1)(n-r-1)-1} n/d.\]
This means that the residual $G/G_h$ action identifies each of the components $A_i/G_h$, so 
\[\hat{Z}_1/G=A_1/G_h \cong \Gr(n,r)/H_{n,r}.\]
This also shows that the orbit of a point in the base of the family $Z$ is precisely the set $\{(c_1,\dots,c_m) \in B: c_i^n=b_i\}$.
\end{proof}
\subsection{Smoothing $Y_{n,r}$ and $Y_{n,r}/G$} 
Recall from Theorem \ref{thm:vgit} that there is a toric variety $Y_{n,r}$ characterized as a blow-up of $X_{P_{n,r}}$ obtained by adding the rays required to take $P_{n,r}$ to $Q_{n,r}$. The variety $Y_{n,r}$ is of interest as it fits into the following diagrams:
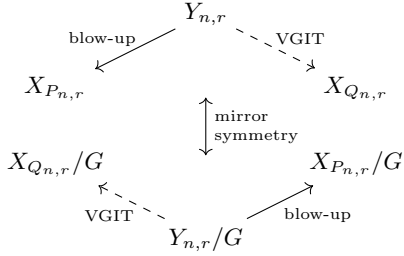
\begin{center}
  \begin{tikzpicture}\small
  \node at (0,0) (A) {$X_{P_{n,r}}$};
  \node at (4,0) (B) {$X_{Q_{n,r}}$};
  \node at (2,1) (C) {$Y_{n,r}$};
  \draw[dashed, ->](C)--(B)  node[align=right,pos=0.25,right] {\tiny \: VGIT};
  \draw[->](C)--(A)  node[pos=0.25,left] {\tiny blow-up \:};
  \node at (2,0) (s) {};
  \node at (2,-1) (t) {};
  \draw[<->](s)--(t) node[align=left,midway,right]{
    \begin{minipage}{1.2cm}
      \tiny mirror symmetry
    \end{minipage}};
  \node at (4,-1) (B1) {$X_{P_{n,r}}/G$};
  \node at (0,-1) (A1) {$X_{Q_{n,r}}/G$};
  \node at (2,-2) (C1) {$Y_{n,r}/G$};
  \draw[->](C1)--(B1)  node[pos=0.1,right] {\tiny \: \: blow-up};
  \draw[->,dashed](C1)--(A1)  node[align=left,pos=0.1,left] {\tiny VGIT \:\:};
   \end{tikzpicture}
\end{center}
$Y_{n,r}$ is the toric variety associated to the fan obtained from the spanning fan of $P_{n,r}$ by blowing up at the rays $v_{\lambda} \in N$, for $\lambda \in \Bnr$ a redundant partition. This depends on the order in which the blow-up is completed. 
 
 Each $v_{\lambda}$ lies in the interior of an intersection of maximal cones of the spanning fan of $P_{n,r}$.  Maximal cones of this fan are \emph{also} indexed by partitions, but not just those in $\Bnr$: there is a maximal cone $C_\mu$ for each $\mu \in \Pnr$ corresponding to the section of $\mathcal{O}(1)$ indexed by $\mu$. 
 \begin{mydef} For each $\lambda \in \Bnr$, let $\Cones(\lambda):=\{\mu \in \Pnr: v_\lambda \in C_\mu\},$ the collection of all maximal cones of the spanning fan of $P_{n,r}$ that contain $v_\lambda$. Fix an order on $\Bnr$. Let $\Bl \PP$ denote the iterated blow-up of $\PP^{\binom{n}{r}-1}$ at the linear subspaces
 \[Z(p_\mu: \mu \in \Cones(\lambda))\]
 where we iterate over $\lambda \in \Bnr$. 
 We define $\Bl \Gr(n,r)$ to be the iterated blow-up of $\Gr(n,r)$ at these same subspaces.
 \end{mydef} 
Consider the Cox ring of of $\Bl \PP$.  In addition to the $p_\mu$,~$\mu \in \Pnr$, there is a generator $x_\lambda$ for each $\lambda \in \Bnr$ graded by the Cartier divisor corresponding to the ray added when blowing up at the $\lambda$th step.  Now instead of blowing up $X_{P_{n,r}}$ using toric geometry, we can view it as a subvariety of $\PP^{\binom{n}{r}-1}$, then compute the blow-up of $X_{P_{n,r}}$ as the main component of its proper transform in the ambient blow-up.

 To write down the binomial equations, we adjust the binomial equations cutting $X_{P_{n,r}}$ out of $\PP^{\binom{n}{r}-1}$. Let $\sigma_1$ and $\sigma_2$ be an incomparable pair, giving rise to the binomial equation
 \[p_{\sigma_1} p_{\sigma_2}-p_{\sigma_1 \vee \sigma_2}p_{\sigma_1 \wedge \sigma_2}.\]
 Although this homogeneous for $\PP^{\binom{n}{r}-1}$, it is not when considered as an equation in the Cox ring of $\Bl \PP$. However, there is a natural minimal way to homogenize this equation by multiplying each factor by some combination of generators $x_\lambda$,~$\lambda \in \Bnr$.   Call the subvariety cut out of $\Bl \PP$ by these adjusted equations $\Bl X_{P_{n,r}}$. 
  \begin{lem} The subvariety $\Bl X_{P_{n,r}}$ is the image of $Y_{n,r}$ under the natural embedding $Y_{n,r} \to \Bl \PP$.
 \end{lem}
 Note that we can lift the action of $G$ on $\Gr(n,r)$ and $X_{P_{n,r}}$ to their blow-ups. 
 \begin{thm} There is a degeneration of $\Bl \Gr(n,r)$ to a possibly disconnected variety whose main component is $Y_{n,r}$.  There is a degeneration of  $\Bl \Gr(n,r)/H_{n,r}$ to a possibly disconnected variety whose main component is $Y_{n,r}/G$. 
 \end{thm}
 \begin{proof}
 The process of homogenization can also be applied to $\Gr(n,r)$ viewed as a subvariety of $\PP^{\binom{n}{r}-1}$ under the Pl\U cker embedding. We can degenerate the homogenized Pl\U cker equations to binomials just as in the Gelfand--Cetlin degeneration: these are the homogenized degenerate Pl\U cker relations cutting out $Y_{n,r}$, except possibly with extra factors $x_\lambda$. This proves the first statement. The proof of the second statement is the same as that of Theorem \ref{thm:Gsmooth}, as $G$ acts trivially on the $x_\lambda$,~$\lambda \in \Bnr$. 
 \end{proof}

 \section{A compactification of the Eguchi--Hori--Xiong superpotential }\label{sec:mirrorproposal}
 In the final section, we give a tentative proposal for a  Calabi--Yau mirror to the Fermat Calabi--Yau hypersurface in the Grassmannian. Eguchi--Hori--Xiong constructed a Laurent polynomial that is a mirror partner to the Grassmannian $\Gr(n,r)$~\cite{eguchi}. It is essentially the toric Hori--Vafa mirror of $X_{P_{n,r}}$. If $y_a$,~$a \in \LQ_1$, are Cox coordinates on $X_{P_{n,r}}$, then the EHX mirror is 
 \[W_q:(\CC^*)^{|\LQ_1|} \to \CC, \quad W_q=\sum_{a \in \LQ_1} y_a\]
 where $q=(q_1,\dots,q_{|\overline{\LQ_0}|})$ are coordinates on the base of the family $\pi:(\CC^*)^{|\LQ_1|} \to (\CC^*)^{|\overline{\LQ_0}|}$.
By the fiber of $W$ we mean $W_q^{-1}(c)$ for some $c$ and $q$. 

 Marsh--Rietsch~\cite{MarshRietsch}  have shown that the classical period of $W$, after setting $q_i = 1$ for all~$i$, equals the regularized quantum period of the Grassmannian; for definitions here, see \cite{CoatesCortiGalkinGolyshevKasprzyk2013}. This implies that it is a \emph{very weak mirror} of the Grassmannian in the sense of \cite{przyjalkowski}. There are stronger notions of mirror symmetry. One important one is the following, due to Przyjalkowski~\cite{przyjalkowski}:
 \begin{mydef} Let $f: (\CC^*)^n \to \CC$ be a Laurent polynomial which is a very weak mirror of a Fano variety $X$. Then $f$ is a \emph{weak mirror} if the fibers of $f$ can be compactified to Calabi--Yau varieties $Z^\vee$.
 \end{mydef}
It is expected that if $Z$ is an anticanonical Calabi--Yau hypersurface in $X$, then $(Z,Z^\vee)$ form a Calabi--Yau mirror pair. This is how Fano and Calabi--Yau mirror symmetry are related when $X$ is a toric variety.

 For our proposed mirror of the Fermat Calabi--Yau hypersurface in the Grassmannian, we do not check mirror-equality of Hodge numbers. However, we show that our mirror is a compactification and smoothing of the EHX superpotential, as one would expect. 
 \begin{mydef} Fix $n$ and $r$ and let $\lambda \in \Pnr$ be a partition. We say that $\lambda$ is a maximal rectangle if it is rectangular and either maximally wide or maximally tall. 
 \end{mydef}
 \begin{rem} The variables $p_\lambda$, where $\lambda$ is maximal rectangular, are the \emph{frozen variables} for the cluster algebra of the Grassmannian (they appear in every cluster chart). 
 \end{rem}
 Let $Z_{n,r}^\vee$ be the family of hypersurfaces cut out of $\Bl \Gr(n,r)/H_{n,r}$ by the equation
 \begin{equation} \label{eq:Znr} \sum_{\lambda \in \Anr} p_\lambda^n + \psi \prod_{p_\mu \text{ frozen}} p_\mu=0.\end{equation}
 Note that this a $H_{n,r}$-equivariant equation, so $Z_{n,r}^\vee$ is well-defined. 
 \begin{thm}\label{thm:compactification} $Z_{n,r}$ has trivial anticanonical class, so is Calabi--Yau in this sense. It is a compactification of the fibers of the Eguchi--Hori--Xiong mirror. 
 \end{thm}
 \begin{proof} We will first show the second part of the statement, and then that equation \eqref{eq:Znr} is a section of the anticanonical line bundle. The adjunction formula then implies that the anticanonical class of the hypersurface is trivial. Let $A_P$ denote the $r(n-r) \times |\LQ_1|$ matrix with integer entries whose columns are given by the rays of the spanning fan of ${P_{n,r}}$ written in the basis $\{b_v\}$. Let $A_{P^\vee}$ denote the $r(n-r) \times \binom{n}{r}$ matrix with integer entries whose columns are given by the rays of the spanning fan of $P_{n,r}^\vee$ written in the basis $\{b_v^\vee\}$. The $|\LQ_1| \times \binom{n}{r}$ matrix 
 \[ A_{P}^t A_{P^\vee} \]
defines a map
\[ (\CC^*)^{\binom{n}{r}} \to (\CC^*)^{|\LQ_1|} \]
by exponentiating. 

This descends to a map between the dense tori in $X_{P_{n,r}}$ and $X_{P_{n,r}^\vee}$. If we write coordinates on the former as $y_a$,~$a \in \LQ_1$, and coordinates on the latter as $z_\lambda$,~$\lambda \in \Pnr$, then this map is
\[(z_\lambda)_{\lambda \in \Pnr} \mapsto \left(\frac{\prod_{a \in \mu }z_\mu^n}{\prod_{\mu\in \Pnr }z_\mu}\right)_{a \in \LQ_1}.\]
This takes the locus $W=-\psi$ to the locus
\[\sum_{a \in \LQ_1} \prod_{a \in \mu }z_\mu^n+\psi \prod_{\mu\in \Pnr }z_\mu=0,\]
which is a Calabi--Yau hypersurface in $X_{P_{n,r}^\vee}=X_{Q_{n,r}}/G$. Note that this equation is $G$-invariant. 

We now apply Theorem \ref{thm:vgit}, and obtain a Calabi--Yau hypersurface in $Y_{n,r}/G$. Viewing $Y_{n,r}$ as a subvariety of $\Bl \PP$, we can choose an identification of Cox rings so that this hypersurface is cut out by the equation 
\[\sum_{\lambda \in \Anr} p_\lambda^n+\psi \prod_{\mu \text{ frozen}}p_\mu=0.\]
It remains Calabi--Yau even as we deform the equations cutting out $Y_{n,r}$, so we get a Calabi--Yau hypersurface in $\Bl \Gr(n,r)/H_{n,r}$.

 \end{proof}
 
\bibliographystyle{amsplain}
\bibliography{bibliography}

\providecommand{\bysame}{\leavevmode\hbox to3em{\hrulefill}\thinspace}
\providecommand{\MR}{\relax\ifhmode\unskip\space\fi MR }
\providecommand{\MRhref}[2]{%
  \href{http://www.ams.org/mathscinet-getitem?mr=#1}{#2}
}
\providecommand{\href}[2]{#2}
\begin{thebibliography}{10}

\bibitem{conifold}
V.~V. Batyrev, I.~Ciocan-Fontanine, B.~Kim, and D.~van Straten, \emph{{Conifold
  transitions and mirror symmetry for Calabi-Yau complete intersections in
  Grassmannians}}, Nucl. Phys. B \textbf{514} (1998), 640--666.

\bibitem{flagdegenerations}
\bysame, \emph{Mirror symmetry and toric degenerations of partial flag
  manifolds}, Acta Mathematica \textbf{184} (2000), no.~1, 1--39.

\bibitem{CiocanFontanineKimSabbah2008}
Ionu\c{t} Ciocan-Fontanine, Bumsig Kim, and Claude Sabbah, \emph{The
  abelian/nonabelian correspondence and {F}robenius manifolds}, Invent. Math.
  \textbf{171} (2008), no.~2, 301--343. \MR{2367022}

\bibitem{CoatesCortiGalkinGolyshevKasprzyk2013}
Tom Coates, Alessio Corti, Sergey Galkin, Vasily Golyshev, and Alexander
  Kasprzyk, \emph{Mirror symmetry and {F}ano manifolds}, European {C}ongress of
  {M}athematics, Eur. Math. Soc., Z\"urich, 2013, pp.~285--300.

\bibitem{crepantconjecture}
Tom Coates, Hiroshi Iritani, and Yunfeng Jiang, \emph{{The Crepant
  Transformation Conjecture for Toric Complete Intersections}}, Advances in
  Mathematics \textbf{329} (2014).

\bibitem{CLS}
D.~A. {Cox}, J.~B. {Little}, and H.~K. {Schenck}, \emph{Toric varieties},
  American Mathematical Society, 2011.

\bibitem{doran}
Charles~F. Doran and John~W. Morgan, \emph{{Mirror symmetry and integral
  variations of Hodge structure underlying one parameter families of Calabi-Yau
  threefolds}}, {Workshop on Calabi-Yau Varieties and Mirror Symmetry}, 5 2005.

\bibitem{eguchi}
Tohru Eguchi, Kentaro Hori, and Chuan-sheng Xiong, \emph{Gravitational quantum
  cohomology}, Int. J. Mod. Phys. \textbf{A12} (1997), 1743--1782.

\bibitem{GL}
N.~Gonciulea and V.~Lakshmibai, \emph{{Degenerations of flag and Schubert
  varieties to toric varieties}}, Transformation Groups \textbf{1} (1996),
  no.~3, 215--248.

\bibitem{greene}
Brian~R. Greene and M.~R. Plesser, \emph{{Duality in {Calabi-Yau} Moduli
  Space}}, Nucl. Phys. B \textbf{338} (1990), 15--37.

\bibitem{gusharpe}
W.~Gu and E.~Sharpe, \emph{A proposal for nonabelian mirrors}, arXiv:1806.04678
  [hep-th], 2018.

\bibitem{horivafa}
Kentaro Hori and Cumrun Vafa, \emph{Mirror symmetry}, arXiv:hep-th/0002222,
  2000.

\bibitem{kalashnikov2}
E.~{Kalashnikov}, \emph{{Laurent polynomial mirrors for quiver flag zero
  loci}}, arXiv e-prints (2019), arXiv:1912.10385.

\bibitem{King1994}
A.~D. King, \emph{Moduli of representations of finite-dimensional algebras},
  Quart. J. Math. Oxford Ser. (2) \textbf{45} (1994), no.~180, 515--530.
  \MR{1315461}

\bibitem{KreuzerSkarke98}
Maximilian Kreuzer and Harald Skarke, \emph{Classification of reflexive
  polyhedra in three dimensions}, Adv. Theor. Math. Phys. \textbf{2} (1998),
  no.~4, 853--871.

\bibitem{manivel}
L.~Manivel and John Swallow, \emph{{Symmetric Functions, Schubert Polynomials,
  and Degeneracy Loci}}, 2001.

\bibitem{MarshRietsch}
B.~Marsh and K.~Rietsch, \emph{{The B-model connection and mirror symmetry for
  Grassmannians}}, Advances in Mathematics \textbf{366} (2020), 107027.

\bibitem{combinatorial}
Ezra Miller and Bernd Sturmfels, \emph{Combinatorial commutative algebra},
  Springer New York, New York, NY, 2005.

\bibitem{przyjalkowski}
V.~V. Przyjalkowski, \emph{{Weak Landau--Ginzburg models for smooth Fano
  threefolds}}, Izv. RAN. Ser. Mat. \textbf{7} (2013), 135--160.

\bibitem{RW}
K.~Rietsch and L.~Williams, \emph{{Newton–Okounkov bodies, cluster duality,
  and mirror symmetry for Grassmannians}}, Duke Math. J. \textbf{168} (2019),
  no.~18, 3437--3527.

\end{thebibliography}
\end{document}